\documentclass[12pt, reqno]{amsart}
\usepackage{amsmath}
\usepackage{amssymb}
\usepackage{amsfonts}
\usepackage{mathrsfs}
\usepackage{bbding}
\usepackage{amssymb,amsmath,amscd,graphicx,latexsym}
\newtheorem{theorem}{Theorem}[section]

\newtheorem{proposition}[theorem]{Proposition}
\newtheorem{lemma}[theorem]{Lemma}

\renewcommand{\a}{\alpha}
\renewcommand{\b}{\beta}
\renewcommand{\d}{\delta}

\newcommand{\f}{\frac}
\newcommand{\g}{\gamma}
\newcommand{\G}{\Gamma}
\renewcommand{\l}{\lambda}
\renewcommand{\L}{\Lambda}
\newcommand{\bea}{\begin{eqnarray}}
\newcommand{\eea}{\end{eqnarray}}
\newcommand{\bna}{\begin{eqnarray*}}
\newcommand{\ena}{\end{eqnarray*}}

\newcommand{\s}{\sigma}
\renewcommand{\th}{\theta}
\newcommand{\fg}{\mathfrak g}
\newcommand{\ve}{\varepsilon}

\begin{document}

\title[The M\"obius function and distal flows] 
{The M\"obius function and distal flows} 
\author{Jianya Liu \& Peter Sarnak}
\address{School of Mathematics
\\
Shandong University
\\
Jinan
\\
Shandong 250100
\\
China}
\email{jyliu@sdu.edu.cn}

\address{Department of Mathematics
\\
Princeton University \& Institute for Advanced Study
\\
Princeton, NJ 08544-1000       
\\
USA}
\email{sarnak@math.princeton.edu}

\date{\today}

\begin{abstract}
We prove that the M\"{o}bius function is linearly disjoint from an 
analytic skew product on the $2$-torus. These flows are distal 
and can be irregular in the sense that their ergodic averages need not 
exist for all points. The previous cases for which such disjointness 
has been proved are all regular. 
We also establish the linear disjointness of 
M\"{o}bius from various distal homogeneous flows. 
\end{abstract}


\maketitle

\addtocounter{footnote}{1}

\section{Introduction} 
\setcounter{equation}{0}
Let $\mathscr{X}=(T, X)$ be a flow, namely $X$ is a compact topological space 
and $T: X\to X$ a continuous map. The sequence $\xi(n)$ 
is observed in $\mathscr{X}$ if there is an $f\in C(X)$ and an $x\in X$,  
such that $\xi(n)=f(T^n x)$. Let $\mu(n)$ be the M\"obius function, that 
is $\mu(n)$  is $0$ if $n$ is not square-free, and is $(-1)^t$ if  $n$ is a 
product of $t$ distinct primes. We say that $\mu$ is linearly 
disjoint from $\mathscr{X}$ if 
\begin{eqnarray}\label{def/DISJO}
\f{1}{N}\sum_{n\leq N} \mu(n)\xi(n) \to 0,  \quad \mbox{as } N\to\infty, 
\end{eqnarray}
for every observable $\xi$ of $\mathscr{X}$. The M\"obius Disjointness Conjecture of 
the second author asserts that $\mu$ is linearly disjoint from every $\mathscr{X}$ 
whose entropy is $0$ \cite{Sar}, \cite{Sar1}. 
The disjointness results for $\mu(n)$ in this paper can be 
proved in the same way for similar multiplicative functions such as 
$\l(n)=(-1)^{\tau(n)}$ where $\tau(n)$ is the number of 
prime factors of $n$.  
This Conjecture has been established for many flows 
$\mathscr{X}$ (see \cite{Dav}, \cite{MauRiv}, \cite{GreTao2}, \cite{BouSarZie}, \cite{Bou})   
however all of these flows are quasi-regular in the sense that 
the Birkhoff averages 
\begin{eqnarray}\label{def/Birkh}
\f{1}{N}\sum_{n\leq N} \xi(n) 
\end{eqnarray}
exist for every  $\xi$ observed in $\mathscr{X}$. 
In fact in all of these cases there is an expected Prime Number Type Theorem of the 
form 
\begin{eqnarray*}
\f{1}{\pi(N)}\sum_{p\leq N} \xi(p) \quad\mbox{converging, } 
\end{eqnarray*}
where $\pi(N)$ is the number of primes less than $N$. Such a 
Prime Number Theorem can certainly fail for irregular flows, that is ones 
for which (\ref{def/Birkh}) fails.     
In this paper we establish some new cases of the Disjointness Conjecture and  
in particular ones which may be irregular.  While the 
flows that we consider can be complicated in terms of the behavior of their individual 
orbits, they are distal and of zero entropy, so that the 
disjointness is still expected to hold. 

Our first result is concerned with certain regular flows, namely affine linear 
maps  of a compact abelian group $X$.  Such a flow $(T, X)$ 
is given by 
\begin{eqnarray}\label{def/AFF}
T(x)=Ax+b
\end{eqnarray}
where $A$ is an automorphism of $X$ and $b\in X$ (see \cite{Hah},  \cite{HoaPar}).  

\begin{theorem}\label{thm1} 
Let $\mathscr{X}=(T, X)$ be an affine linear flow on a compact 
abelian group which is of zero entropy. Then $\mu$ 
is linearly disjoint from $\mathscr{X}$.  
\end{theorem}

The flows in Theorem~\ref{thm1} are distal quasi-regular,  
and our main 
result is concerned with nonlinear distal flows on such spaces. 
We restrict to $X={\Bbb T}^2$ the two dimensional torus 
${\Bbb R}^2/{\Bbb Z}^2$ and consider nonlinear smooth (or even analytic) 
skew products as discussed in Furstenberg \cite{Fur61}. $T: {\Bbb T}^2\to {\Bbb T}^2$ 
is given by 
\begin{eqnarray}\label{def/SKEW}
T(x, y)= (ax+\a, cx+dy+ h(x))
\end{eqnarray}
where $a, c, d\in \Bbb Z, ad=\pm 1,  \a\in \Bbb R$ 
and $h$ is a smooth periodic function of period $1$. 
The affine 
linear part is in the form 
$$
\left[
\begin{array}{ccc}
a & 0\\
c & d
\end{array}
\right] \in GL_2 (\Bbb Z), 
$$
ensuring that $T$ has zero entropy (and it can always be brought into 
this form). The flow $(T, {\Bbb T}^2)$ is distal and this skew product 
is a basic building block (with $e(h(x))$ continuous) in Furstenberg's 
classification theory of minimal distal flows \cite{Fur63}. Thus 
our main result takes the first step towards handling distal flows, 
by dealing with the new dynamical complexities that are present in 
such skew products. 
If $\a$ is diophantine, that is 
$$
\bigg|\a-\f{a}{q}\bigg|\geq \f{c}{q^m} 
$$
for some $c>0, m<\infty$ and all $a/q$ rational,  then $T$ can be conjugated 
by a smooth map of ${\Bbb T}^2$ to its affine linear part 
\begin{eqnarray}\label{DIO/conjugate}
(x, y)\mapsto (ax+\a, cx+dy+ \b)
\end{eqnarray}
where 
$$
\b=\int_0^1 h(x)dx
$$ 
(see \cite{SanUrz}). 
Hence the disjointness of $\mu$ from 
$\mathscr{X}=(T, {\Bbb T}^2)$ for a $T$ with a diophantine 
$\a$, follows from Theorem~\ref{thm1}.  
However if $\a$ is not diophantine the dynamics of the 
flow $(T, {\Bbb T}^2)$ can be very different from an affine 
linear flow. For example, as Furstenberg shows it may be 
irregular (i.e. the limits in (\ref{def/Birkh}) fail to exist  for certain observables, 
see (\ref{qk+1>eqk}) and (\ref{Furs/h=})).  
Nevertheless our main result shows that  these nonlinear skew products 
are linearly disjoint from $\mu$, at least if $h$ satisfies some further small  
technical hypothesis. Firstly we assume that $h$ is analytic, namely that 
if 
\begin{eqnarray}\label{h=/FOUR}
h(x)=\sum_{m\in \Bbb Z} \hat{h}(m) e(mx)
\end{eqnarray}
then 
\begin{eqnarray}\label{hhat/UPP}
\hat{h}(m)\ll e^{-\tau |m|}
\end{eqnarray}
for some $\tau>0$. Secondly we assume that there is 
$\tau_2<\infty$ such that 
\begin{eqnarray}\label{hhat/LOW}
|\hat{h}(m)|\gg e^{-\tau_2 |m|}.  
\end{eqnarray}
This is not a very natural condition being an artifact of our proof. 
However it is not too restrictive and the following applies rather 
generally (and most importantly there is no condition on $\a$). 

\begin{theorem}\label{thm2} 
Let $\mathscr{X}=(T, {\Bbb T}^2)$ be of the form (\ref{def/SKEW}), 
with $h$ satisfying (\ref{hhat/UPP}) and (\ref{hhat/LOW}). 
Then $\mu$ is linearly disjoint from $\mathscr{X}$.  
\end{theorem}

Theorem~\ref{thm1} deals with the affine 
linear distal flows on the $n$-torus. A different source 
of  homogeneous quasi-regular distal flows are the affine linear flows on nilmanifold 
$X=G/\G$ where $G$ is a nilpotent Lie group and 
$\G$ a lattice in $G$.  For 
$\mathscr{X}=(T, G/\G)$ where $T(x)=\a x\G$  
with $\a\in G$, i.e. translation on $G/\G$, the 
linear disjointness of $\mu$ and $\mathscr{X}$ is proven in 
\cite{GreTao1} and \cite{GreTao2}. Using the classification of zero entropy (equivalently distal) 
affine linear flows on nilmanifolds \cite{Dan}, and 
Green and Tao's results  we prove 

\begin{theorem}\label{thm3} 
Let $\mathscr{X}=(T, G/\G)$ where $T$ is an affine 
linear map of the nilmanifold  $G/\G$ of zero entropy. 
Then $\mu$ is linearly disjoint from $\mathscr{X}$.  
\end{theorem}

We end the introduction with brief outline of the paper and proofs. Theorem~\ref{thm1} 
with a rate of convergence is proved in \S2. We first reduce to the torus case and then handle 
the torus case by Fourier analysis and classical results of Davenport and 
Hua on exponential sums concerning the 
M\"{o}bius function, which is stated as Lemma~\ref{lem:VDH} in the present paper. 
The proof of Theorem~\ref{thm2} occupies \S\S3-6. The assertion of 
Theorem~\ref{thm2} holds for all $\a$, and so we have to consider all 
diophantine possibilities of $\a$. The case when $\a$ is rational is easy and 
this is done in \S3. When $\a$ is irrational we have to distinguish three cases 
(A), (B), and (C), and the first two cases with rates of convergence 
are handled in \S4 and \S5 respectively via different analytic techniques.  
The most complicated case (C) is studied in \S6, 
and the tool for this is the Bourgain-Sarnak-Ziegler 
finite version of the Vinogradov method (see Lemma~\ref{lem:BSZ})
\footnote
{Earlier implicit versions of this can be found in the literature, for example in \cite{Kat}.},  
incorporated with various analytic methods 
such as Poisson's summation and stationary phase. Thus in case (C) we offer 
no rate. A sub-case of (C) requires an analogue for trigonometric polynomials 
in place of polynomials of the uniform cancellation in Hua's Lemma~\ref{lem:VDH}, 
see Proposition~\ref{prop:TriHua}.  Furstenberg \cite{Fur61} gives examples of skew product transformations 
of the form (\ref{def/SKEW}) which are not regular in the sense of (\ref{def/Birkh}). 
Many of the flows ${\mathscr X}$ 
in Theorem~\ref{thm2} have this property and we show in \S7 that Furstenberg's 
examples are smoothly conjugate to such ${\mathscr X}$'s. In particular his examples are linearly 
disjoint from $\mu$.  By analyzing the structure of affine linear maps of nilmanifolds, 
Theorem~\ref{thm3} is reduced  in \S8 to a recent result of Green-Tao  
of polynomial orbits on nilmanifolds (see Lemma~\ref{lem:4-1}).  
 
Throughout the paper there are various double exponential functions like 
$e(e(f(n)))$ against the M\"{o}bius function $\mu(n)$ 
where $e(x)=e^{2\pi i x}$ as usual, and so we have to keep track of 
the dependence of each parameter very carefully. 

\section{Theorem~\ref{thm1}} 
\setcounter{equation}{0}

\subsection{Reduction to the toral case}  
We first reduce to the case that $X$ is a torus (not necessarily connected), that 
is $X={\Bbb T}^r\times C={\Bbb R}^r/{\Bbb Z}^r\times C$ for some integer 
$r\geq 0$ and $C$ is a finite (abelian) group. Since the linear combinations of 
characters $\psi\in \Gamma:=\widehat{X}$, the (discrete) dual group 
of $X$, are dense in $C(X)$ it suffices to show that 
\begin{eqnarray}\label{Thm1/equiv}
\f{1}{N}\sum_{n\leq N} \mu(n) \psi(T^n x)\to 0, \quad \mbox{as } N\to\infty  
\end{eqnarray}
for every fixed $x\in X$ and $\psi\in \Gamma$. So fix 
$\psi\in \Gamma$ and let $C_\psi$ be the smallest closed 
subgroup of $\Gamma$ containing $\psi$ and invariant by $A$. 
Here we are denoting by $T$ the affine linear map $Tx=Ax+b$ of 
$X$ and $A$ acts on $\Gamma$ by $A\phi(x)=\phi (Ax)$ for 
$\phi\in \Gamma, x\in X$. If $C$ is a subgroup of $\Gamma$ 
let $C^\perp$ the annihilator of $C$ be the closed subgroup 
of $X$ given by $C^\perp=\{x\in X: c(x)=1 \mbox{ for all } c\in C\}$.  
Set $X_\psi$ to be the compact quotient group $X/C_\psi^\perp$. 
By definition $\widehat{X/C_\psi^\perp} = C_\psi$. For $x\in X$ and $y\in C_\psi^\perp$, 
\begin{eqnarray*}
T(x+y)=A(x+y)+b=Ax+b+Ay\equiv Ax+b \bmod C_\psi^\perp 
\end{eqnarray*}
since $C_\psi$ is $A$-invariant. Hence $T(x+y)=Tx\bmod C_\psi^\perp$, that 
is $T$ induces an affine linear map $T_\psi$ of $X_\psi$. Put 
another way the flow $\mathscr{X}_\psi=(T_\psi, X_\psi)$ 
is a factor of $\mathscr{X}=(T, X)$. Since we are assuming that 
$\mathscr{X}$ has zero entropy it follows that so does $\mathscr{X}_\psi$. 
This in turn implies that $C_\psi$ is finitely generated, as shown by 
Aoki (see \cite{Aok} page 13). Being the dual of 
$X_\psi$ it follows that 
$X_\psi$ is isomorphic to  
${\Bbb T}^r\times C$ for some  
$r\geq 0$ and finite $C$. Moreover for $n\geq 0$ 
\begin{eqnarray*}
\psi(T^n x)=\widetilde{\psi}(T_\psi^n \dot{x}) 
\end{eqnarray*}
where in the last $\dot{x}$ is the projection of $x$ in 
$X_\psi$ and $\widetilde{\psi}$ is the character on 
$X_\psi$ induced by $\psi$. In particular the observable 
$\psi(T^n x)$ on $\mathscr{X}$ is equal to $\widetilde{\psi}(T_\psi^n \dot{x})$ 
on $\mathscr{X}_\psi$. Thus (\ref{Thm1/equiv}) will follow from the 
linear disjointness of the M\"{o}bius function from $\mathscr{X}_\psi$. This 
completes the reduction to the toral case. 

\subsection{Affine linear maps on a torus}   
We have reduced Theorem~\ref{thm1} to the case that 
${\mathscr X}=(T, X)$ with $X={\Bbb R}^r/{\Bbb Z}^r\times C$ 
with $C$ finite and $T$ in the form (\ref{def/AFF}) and of zero entropy. For our purpose of 
examining observables $\xi(n)$ in this flow, we can ``linearize" the flow by 
doubling the number of variables. That is consider $Y=X\times X$ and the 
linear automorphism $W$ given by 
\begin{eqnarray}\label{def/W=} 
W(x_1, x_2)=(Ax_1+x_2, x_2).  
\end{eqnarray}
${\mathscr Y}=(W, Y)$ is clearly of zero entropy since $\mathscr X$ is so, and the 
orbit $W^n (x_1, b)$ is equal to $(T^n x_1, b), n\geq 1$. Hence it suffices to 
prove Theorem~\ref{thm1} for such $\mathscr Y$'s. That is we can assume that 
${\mathscr X}=(W, X)$ with $X={\Bbb R}^m/{\Bbb Z}^m\times F$, $F$ finite 
and $W$ is a linear automorphism of $X$ of zero entropy. Either by noting that 
the induced action of $W$ on $\widehat{X}$ must preserve $1\times F$ 
(since these are precisely the elements of finite order in $\widehat{X}$) or using the 
continuity of $W$ to conclude that it preserves the connected component of $0$ in $X$ 
(i.e. ${\Bbb R}^m/{\Bbb Z}^m\times \{0\}$), we see that $W$ takes 
the block triangular form 
\begin{eqnarray}
W(\th, f)=(B\th+Cf, Df) 
\end{eqnarray}
where $B: {\Bbb R}^m/{\Bbb Z}^m \to {\Bbb R}^m/{\Bbb Z}^m$ is an automorphism 
of this (connected) torus, $C: F\to {\Bbb R}^m/{\Bbb Z}^m$ is a homomorphism and 
$D: F\to F$ is an automorphism of $F$. The automorphism $B$ lifts to a linear 
automorphism $\widetilde{B}$ of ${\Bbb R}^m$ which preserves ${\Bbb Z}^m$, 
so that $\widetilde{B}\in GL_m (\Bbb Z)$. Since $W$ has zero entropy so does $B$ 
and it is known that this implies that $\widetilde{B}$ is quasi-unipotent \cite{Dan}.  
That is, for some $\nu_1\geq 1, \widetilde{B}^{\nu_1}=U$ is unipotent, 
or $U=I+N_1$ with $N_1$ nilpotent and $I$ the identity matrix. 
Also since $F$ is finite it is clear that $D^{\nu_2}=I$ for  some 
$\nu_2\geq 1$. Let $\nu= \mbox{lcm}(\nu_1, \nu_2)$. Then we  have that 
\begin{eqnarray}
W^\nu (\th, f) = ((I+N_1)\th+C_1 f, f) 
\end{eqnarray}
where $C_1$ is a morphism from $F$ to ${\Bbb R}^m/{\Bbb Z}^m$. In particular 
\begin{eqnarray}
\Phi:=W^\nu= I+N 
\end{eqnarray}
where $N: X\to X$ satisfies $N^{k+1}\equiv 0$ for some $k\geq 0$. 
Thus for $q\geq 0$ an integer 
\begin{eqnarray}
\Phi^q 
=\sum_{t=0}^q {q \choose t}N^t 
=\sum_{t=0}^{\min (k,q)} 
{q \choose t}N^t.  
\end{eqnarray}
Writing $n\geq 0$ as $n=q \nu +l$ with $0\leq l<\nu$ we have 
\begin{eqnarray*}
W^n = W^{q \nu +l} =\Phi^q W^l 
\end{eqnarray*}
and hence if $x\in X$ and $n=q\nu+l$ then 
\begin{eqnarray}
W^n x  
=\sum_{t=0}^{\min (k,q)} {q \choose t}N^t W^l x
=\sum_{t=0}^{\min (k,q)} {q \choose t} \xi_{l, t}
\end{eqnarray}
where 
\begin{eqnarray}
\xi_{l,t}=N^t W^l x. 
\end{eqnarray}
For $q$ varying, $q\geq k$ and $\psi\in \widehat{X}$ fixed we have 
\begin{eqnarray}\label{Wqnu+l=}
\psi(W^{q \nu +l} x)   
&=& \psi\bigg(\sum_{t=0}^k {q \choose t} \xi_{l, t}\bigg) \nonumber\\ 
&=& \psi (\xi_{l, 0})
\psi (\xi_{l, 1})^{{q \choose 1}}
\psi (\xi_{l, 2})^{{q \choose 2}}
\cdots \psi (\xi_{l, k})^{{q \choose k}}. 
\end{eqnarray}
The character $\psi\in\widehat{X}$ has the form 
$\psi: x\mapsto e(\langle v, x\rangle)$ for some 
$v=(v_1, \ldots, v_m)\in {\Bbb Z}^m$ where $\langle v, x\rangle$ means the 
dot product in ${\Bbb R}^m$, and hence the right-hand side 
of (\ref{Wqnu+l=}) is $e(Y(q))$ where $Y(q)$ is a polynomial in $q$ with 
degree $\leq k$ and with coefficients depending on $v$ and the $\xi$'s.  Changing variables 
from $q$ to $n$ by $n=\nu q+l$ with $0\leq l\leq \nu-1$, we see that 
$Y(q)=\phi(n)$ a polynomial in $n$ with degree $\leq k$ and coefficients 
depending on $v, \nu, l$ and the $\xi$'s.  It follows that 
\begin{eqnarray}\label{muWnx}
\sum_{n\leq N}\mu(n)\psi (W^n x)
&=&\sum_{l=0}^{\nu-1}\sum_{n\leq N\atop n\equiv l(\bmod \nu)}
\mu(n)\psi (W^n x)\nonumber\\
&=&\sum_{l=0}^{\nu-1}\sum_{n\leq N\atop n\equiv l(\bmod \nu)}\mu(n)
e(\phi(n)). 
\end{eqnarray}

Theorem~\ref{thm1} for $(W, X)$ now follows from the following classical 
result proved by Davenport \cite{Dav} for $\phi$ linear and 
by Hua \cite{Hua} for $\phi$ nonlinear. This lemma will also be used in later sections. 

\begin{lemma}\label{lem:VDH}
Let $\nu$ be a positive integer and $0\leq l<\nu$.   
Let 
$$
\phi(u)=\a_d u^d + \a_{d-1} u^{d-1}+\cdots +\a_1 u +\a_0
$$
be a real polynomial of degree $d>0$. Then, for arbitrary $A>0$,
\begin{eqnarray}\label{set/VDH}
\sum_{n\leq N\atop n\equiv l(\bmod \nu)} 
\mu(n)e(\phi(n))
\ll \f{N}{\log^A N}
\end{eqnarray}
where the implied constant may depend on $A$ and $\nu,$     
but is independent of any of the coefficients $\a_d,\ldots,\a_0$.
\end{lemma}

This can be established 
by Vinogradov's method or its modern variants, such as Vaughan's identity 
or Heath-Brown's identity.  The estimate (\ref{set/VDH}), with $\mu$ replaced by 
$\L$ the von Mangoldt function, 
was established in Hua \cite{Hua}, Theorem~10.  

\section{Theorem~\ref{thm2} with $\a$ rational}   
\setcounter{equation}{0}

\subsection{Reduction}   
Without loss of generality we may assume that $a=d=1$ in (\ref{def/SKEW}). 
Thus 
\begin{eqnarray}\label{def/SimS}
T: \ (x_1, x_2) \mapsto (x_1+\a, cx_1+ x_2 + h(x_1)), 
\end{eqnarray} 
where $c\in \Bbb Z, \a\in \Bbb R$ and $h$ is a smooth periodic function of period $1$. 
Since the linear combinations of 
characters $\psi\in \widehat{\Bbb T}^2$ are dense in $C({\Bbb T}^2)$, 
it is sufficient to show that 
\begin{eqnarray*}\label{Sum/mufSnx}
\sum_{n\leq N} \mu(n) \psi(T^n x)=o(N),  
\quad \mbox{as } N\to\infty  
\end{eqnarray*}
for any fixed $x\in X$ and any fixed $\psi\in \widehat{\Bbb T}^2$. 
Note that any $\psi\in \widehat{\Bbb T}^2$ has the form 
$\psi: x\mapsto e(\langle b, x\rangle)$ for some $b=(b_1, b_2)\in {\Bbb Z}^2$ 
where $\langle b, x\rangle$ means the dot product in ${\Bbb R}^2$.  
Applying (\ref{def/SimS}) repeatedly, we have  
$T^n: (x_1,x_2)\mapsto (y_1(n), y_2(n))$ 
with
\begin{eqnarray}\label{y1=}
y_1(n)=x_1+n\a, 
\end{eqnarray}
\begin{eqnarray}\label{y2=}
y_2(n)= c \f{n(n-1)}{2}\a + c n x_1 + x_2+ \sum_{j=0}^{n-1} h(x_1+j\a). 
\end{eqnarray}
It follows that 
\begin{eqnarray*}
\langle b, y(n)\rangle
=b_1y_1(n)+b_2y_2(n)
=P(n) + b_2 \sum_{j=0}^{n-1} h(x_1+j\a), 
\end{eqnarray*}
where 
\begin{eqnarray}\label{def/Pnax}
P(n)= b_1(x_1+n\a) +b_2 \bigg(c \f{n(n-1)}{2}\a + c n x_1+x_2\bigg),  
\end{eqnarray}
a polynomial of $n$ with degree at most $2$ and with coefficients depending on
$\a, x_1, c,$ and $b$. Put 
\begin{eqnarray}\label{T2/S=See}
S(N)
&=&\sum_{n\leq N}\mu(n)e(\langle b, y(n)\rangle)  \nonumber\\ 
&=&\sum_{n\leq N}\mu(n)e\bigg(P(n)+ b_2 \sum_{j=0}^{n-1} h(x_1+j\a)\bigg). 
\end{eqnarray} 
Then the aim is to prove that 
\begin{eqnarray}\label{Sum/muee} 
S(N)=o(N) 
\end{eqnarray} 
for any fixed $x=(x_1, x_2) \in {\Bbb T}^2$ and any fixed 
$b=(b_1, b_2) \in {\Bbb Z}^2$, which will be done 
in \S\S3-6.  We may suppose that $b_2\not=0$ since otherwise (\ref{Sum/muee})  
follows from Lemma~\ref{lem:VDH} with $\nu=l=1$ immediately. 

Some of our results in \S\S3-6 actually hold for any smooth periodic $h$, not necessarily analytic. 
Suppose that $h:\Bbb R\to\Bbb R$ is a smooth periodic function with period $1$. 
Then it has the Fourier expansion
\begin{eqnarray}\label{h/Ex/Ra}
h(x)=\sum_{m\in \Bbb Z} \hat{h}(m)e(mx),
\end{eqnarray}
which converges absolutely and uniformly on $\Bbb R$, and its coefficients $\hat{h}(m)$
satisfy
\begin{eqnarray}\label{hath/Smo}
\hat{h}(m)\ll_A (|m|+2)^{-A}
\end{eqnarray}
for arbitrary $A>0$. 
We can transform $S(N)$ by inserting the Fourier expansion (\ref{h/Ex/Ra}) of $h$. Thus,  
\begin{eqnarray*}\label{SUMhjx}
\sum_{j=0}^{n-1} h(x_1+j\a)
&=&\sum_{m\in\Bbb Z}\hat{h}(m)e(mx_1) \sum_{j=0}^{n-1} e(jm\a) \nonumber\\  
&=&\sum_{m\in\Bbb Z} \hat{h}(m)e(mx_1) \f{e(nm\a)-1}{e(m\a)-1}
\end{eqnarray*}
where we understand that
\begin{eqnarray}\label{=n/conv}
\f{e(nm\a)-1}{e(m\a)-1}=n \quad \mbox{for } m\a\in\Bbb Z.
\end{eqnarray} 
This can happen only when $\a$ is rational. 
It follows that 
\begin{eqnarray}\label{T2/S/Simp}
S(N)=\sum_{n\leq N}\mu(n)e\bigg(P(n)
+ b_2 \sum_{m\in\Bbb Z} \hat{h}(m)e(mx_1) \f{e(nm\a)-1}{e(m\a)-1}\bigg). 
\end{eqnarray} 

\subsection{The case of rational $\a$}  

In this section we establish (\ref{Sum/muee}) for rational $\a$. 
We remark that in this case the Fourier expansion (\ref{h/Ex/Ra}) 
of $h$ is not necessary. 

\begin{proposition}\label{prop/rat}
Let $S(N)$ be as in (\ref{T2/S=See}), and 
$h:\Bbb R\to\Bbb R$ a smooth periodic function with period $1$. 
If $\a \in \Bbb Q$ then 
\begin{eqnarray}\label{Sum/muee/Rat}
S(N)\ll N\log^{-A} N, 
\end{eqnarray}  
where $A>0$ is arbitrary, and the implied constant depends on $A$ and $\a$ 
only.  
\end{proposition}

\begin{proof} 
If $\a=0$ then (\ref{T2/S/Simp}) becomes 
\begin{eqnarray*}
S(N)=\sum_{n\leq N}\mu(n)e\{P(n)+ b_2 n h(x_1)\},  
\end{eqnarray*} 
and the desired result follows directly from Lemma~\ref{lem:VDH}. 

Suppose $\a=l/q$ with $(l,q)=1$. We start from (\ref{T2/S=See})  
and split the sum over $j$ there into residue classes modulo $q$. Since 
$h$ has period $1$,  
\begin{eqnarray*}
\sum_{j=0}^{n-1} h(x_1+j\a)
= \sum_{j_0=0}^{q-1} h\bigg(x_1+ \f{j _0 l}{q}\bigg) \bigg[\f{n-j_0}{q}\bigg]. 
\end{eqnarray*}
If $n\equiv k(\bmod q)$ for some $0\leq k\leq q-1$, 
then the above equals  
$\g_1\f{n-k}{q}+ \g_2(\f{n-k}{q}-1)$ where 
\begin{eqnarray*}
\g_1 = \sum_{j_0=0}^{k} h\bigg(x_1+ \f{j _0 l}{q}\bigg), \quad   
\g_2 = \sum_{j_0=k+1}^{q-1} h\bigg(x_1+ \f{j _0 l}{q}\bigg). 
\end{eqnarray*}
It follows from this and (\ref{T2/S=See}) that 
\begin{eqnarray*}\label{S/=See/con}
S(N)
&=&\sum_{k =0}^{q-1} 
\sum_{n\leq N\atop n\equiv k (\bmod q)} 
\mu(n)e\bigg(P(n)+ b_2 \g_1\f{n-k}{q}+ b_2 \g_2 \bigg(\f{n-k}{q}-1\bigg) \bigg) \\ 
&\ll& \f{N}{\log^A N}
\end{eqnarray*} 
by Lemma~\ref{lem:VDH}, where the implied constant depends on $A$ and $q$ only.  
This proves the proposition.    
\end{proof} 

\section{The continued fraction expansion of $\a$}  
\setcounter{equation}{0}
\subsection{The continued fraction expansion of $\a$.} 
From now on we assume that $\a$ is irrational, and our argument will 
depend on the continued fraction expansion of $\a$. 
Every real number $\a$ has its continued fraction 
representation 
\begin{eqnarray}\label{a=a0a1+}
\a=a_0+\f{1}{a_1+\f{1}{a_2+\cdots}}
\end{eqnarray}
where $a_0=[\a]$ is the integral part of $\a$, 
and $a_1, a_2, \ldots$ are positive integers. 
The expression (\ref{a=a0a1+}) is infinite since $\a\not\in \Bbb Q$.   
We write $[a_0; a_1, a_2, \ldots]$ for the 
expression on the right-hand side of (\ref{a=a0a1+}), which is the limit of 
the finite continued expressions 
\begin{eqnarray}\label{a=a0a1F}
[a_0; a_1, a_2, \ldots, a_k]=a_0+\f{1}{a_1+\f{1}{a_2+\cdots+\f{1}{a_k}}}
\end{eqnarray}
as $k\to \infty$. 
Writing 
\begin{eqnarray*}
\f{l_k}{q_k}=[a_0; a_1, a_2, \ldots, a_k], 
\end{eqnarray*}
we have $l_0=a_0, l_1=a_0a_1+1, q_0=1, q_1=a_1, $ 
and for $k\geq 2$, 
\begin{eqnarray*}
l_k=a_k l_{k-1}+l_{k-2}, \quad 
q_k=a_k q_{k-1}+q_{k-2}.  
\end{eqnarray*}
Since $\a$ is irrational we have $q_{k+1}\geq q_k+1$ for all $k\geq 1$. 
An induction argument gives the stronger 
assertion that 
$q_k\geq 2^{(k-1)/2}$ 
for all $k\geq 2$, and thus $q_k$ increases at least like an exponential function of $k$.  
The irrationality of $\a$ also implies that, for all $k\geq 2$, 
\begin{eqnarray}\label{ratAPP}
\f{1}{2 q_k q_{k+1}}<\bigg|\a-\f{l_k}{q_k}\bigg| <\f{1}{q_kq_{k+1}}, 
\end{eqnarray} 
which will be used in our later argument. 

Let ${\mathcal Q}$ be the set of all $q_k$ with $k=0, 1, 2, \ldots$; 
note that $q_0=1$.  
Sometimes it is convenient to abbreviate $q_k$ to $q$, and 
$q_{k+1}$ to $q^+$. Let $B$ be a large positive constant to be decided later. 
The set 
${\mathcal Q}$ can be partitioned as 
${\mathcal Q}^{\flat}\cup {\mathcal Q}^{\sharp}$ where 
\begin{eqnarray}\label{Q>Q<}
{\mathcal Q}^{\flat}=\{1\}\cup \{q\in {\mathcal Q}: q^+ \leq q^B\}, 
\quad 
{\mathcal Q}^{\sharp}=\{q\in {\mathcal Q}: q^+ > q^B \mbox{ and } q\geq 2\}. 
\end{eqnarray}

\begin{lemma}\label{lem:2ser}
Let $h:\Bbb R\to\Bbb R$ be a smooth periodic function
of period $1$. Then the following two series 
\begin{eqnarray}\label{Ser/1/Q}
\sum_{q\in {\mathcal Q}}
\sum_{q\leq |m| < q^+ \atop q\nmid m} 
\f{|\hat{h}(m)|}{\|m\a\|}, 
\quad 
\sum_{q\in {\mathcal Q}^{\flat}} 
\sum_{q\leq |m| < q^+ \atop q|m} 
\f{|\hat{h}(m)|}{\|m\a\|}
\end{eqnarray}
are convergent.   
\end{lemma}

\begin{proof} We just handle positive $m$; proof for negative $m$ 
is the same. We first establish the convergence of the first series in 
(\ref{Ser/1/Q}). By 
(\ref{ratAPP}), $\a$ can be written in the form 
$$
\a=\f{l}{q}+\f{\g}{q(q+1)}, \quad q\in {\mathcal Q}, \ (l,q)=1, \ |\g|<1. 
$$
Therefore, for $m=1, 2, \ldots, q-1$,  
\begin{eqnarray*}
m\a
=\f{m l}{q}+\f{m\g}{q(q+1)} 
=\f{m l}{q}+\f{\g'}{q+1}, \quad |\g'|<1, 
\end{eqnarray*} 
and hence 
\begin{eqnarray*}
\sum_{m=1}^{q-1} \f{1}{\|m\a\|}
=\sum_{m=1}^{q-1} \f{1}{\|\f{m l}{q}+\f{\g'}{q+1}\|}. 
\end{eqnarray*} 
We write $m l\equiv r(\bmod q)$ with $1\leq |r|\leq q/2$, so that 
the last denominator is 
\begin{eqnarray*}
\geq \f{|r|}{q}-\f{|\g'|}{q+1}=\f{(q+1)|r|-q|\g'|}{q(q+1)}
\geq \f{|r|}{q(q+1)}, 
\end{eqnarray*} 
and consequently 
\begin{eqnarray*}
\sum_{m=1}^{q-1} \f{1}{\|m\a\|}
\ll \sum_{1\leq r\leq q/2} \f{q(q+1)}{r}
\ll q(q+1)\log q.  
\end{eqnarray*} 
It follows that, for any positive $t$, 
\begin{eqnarray}\label{posh/t}
\sum_{1\leq m\leq t\atop q\nmid m} \f{1}{\|m\a\|}
\ll \bigg(\f{t}{q}+1\bigg)q^2\log q. 
\end{eqnarray} 
By (\ref{hath/Smo})  and partial integration, 
\begin{eqnarray*}
&& \sum_{q\leq m< q^+\atop q\nmid m} \f{|\hat{h}(m)|}{\|m\a\|}
\ll \sum_{q\leq m< q^+\atop q\nmid m} \f{m^{-A}}{\|m\a\|} 
\ll \int_q^{\infty} t^{-A} 
d\bigg\{\sum_{1\leq m\leq t\atop q\nmid m} \f{1}{\|m\a\|}\bigg\} \\ 
&& 
\quad \ll q^2\log q \int_q^{\infty} t^{-A}\bigg(\f{t}{q}+1\bigg)dt  
\ll_A q^{-A+3}\log q,  
\end{eqnarray*} 
and hence the first series in (\ref{Ser/1/Q}) is convergent. 
 
Next we consider the second series in (\ref{Ser/1/Q}). We assume $q>1$ since the 
case $q=1$ can be easily checked.  
Again by (\ref{ratAPP}),    
$$
\f{m}{2q q^+}< \bigg|m\a- m\f{l}{q}\bigg|< \f{m}{q q^+}. 
$$
Since $q|m$, we may write $m=m' q$, and hence the above becomes 
$$
\f{m'}{2 q^+}< \|m\a\|< \f{m'}{q^+}.  
$$
It follows that 
\begin{eqnarray}\label{qq+q|m}
\sum_{q\leq m< q^+ \atop q|m}\f{1}{ \|m\a\|}
\leq \sum_{m'\leq q^+/q} \f{2q^+}{m'}
\ll q^+\log q^+,   
\end{eqnarray}
and the last term is $\ll q^B\log (q^B)$ since $q\in {\mathcal Q}^{\flat}$. 
From this and (\ref{hath/Smo}) we deduce that 
\begin{eqnarray*}
\sum_{q\leq m< q^+ \atop q|m}\f{|\hat{h}(m)|}{ \|m\a\|}
\ll q^{-A+B}\log (q^B),    
\end{eqnarray*} 
which proves that second series in (\ref{Ser/1/Q}) is also convergent. 
The lemma is proved. 
\end{proof}

\subsection{Transformation of the sum $S(N)$.}  
Lemma~\ref{lem:2ser} can be used to understand the sum over $m$   
in (\ref{T2/S/Simp}); it implies that the following two series 
\begin{eqnarray}\label{Ser/1st/}
\sum_{q\in {\mathcal Q}}
\sum_{q\leq |m|< q^+ \atop q\nmid m} 
\hat{h}(m)e(mx_1) \f{e(nm\a)}{e(m\a)-1}
\end{eqnarray}
and 
\begin{eqnarray}\label{Ser/2nd/}
\quad \sum_{q\in {\mathcal Q}^{\flat}} 
\sum_{q\leq |m|< q^+ \atop q|m}\hat{h}(m)e(mx_1) \f{e(nm\a)}{e(m\a)-1} 
\end{eqnarray}
are absolutely convergent. Denote by $g(n\a+x_1)$ the sum of these two series,  that is 
\begin{eqnarray*}\label{DVD}
g(n\a+x_1)
=\bigg\{\sum_{q\in {\mathcal Q}}
\sum_{q\leq |m|< q^+ \atop q\nmid m} 
+\quad \sum_{q\in {\mathcal Q}^{\flat}} 
\sum_{q\leq |m|< q^+ \atop q|m}\bigg\}
\hat{h}(m)e(mx_1) \f{e(nm\a)}{e(m\a)-1},   
\end{eqnarray*} 
where $g:\Bbb R\to\Bbb R$ is a  
smooth periodic function of period $1$.  
It follows that 
\begin{eqnarray*}
\bigg\{\sum_{q\in {\mathcal Q}}
\sum_{q\leq |m|< q^+ \atop q\nmid m} 
+\quad \sum_{q\in {\mathcal Q}^{\flat}} 
\sum_{q\leq |m|< q^+ \atop q|m}\bigg\}
\hat{h}(m)e(mx_1) \f{e(nm\a)-1}{e(m\a)-1}  
=g(n\a+x_1)
-g(x_1). 
\end{eqnarray*} 
\medskip 
Therefore the sum over $m$ in (\ref{T2/S/Simp}) can be 
written as 
\begin{eqnarray}\label{n1+ON+g-g}
g(x_1+n\a)-g(x_1)+ H(x)
\end{eqnarray}
with 
\begin{eqnarray}\label{Phix=}
H(x)=\sum_{q\in {\mathcal Q}^{\sharp}}
\sum_{q\leq |m|<q^+ \atop q|m} 
\hat{h}(m)e(mx_1) \f{e(xm\a)-1}{e(m\a)-1}. 
\end{eqnarray}
Inserting these into (\ref{T2/S/Simp}), we have 
\begin{eqnarray*}\label{See/con2}
S(N)=e(-b_2 g(x_1)) 
\sum_{n\leq N}\mu(n)e\{P(n)+b_2 g(n\a+x_1)+ b_2 H(n)\}
\end{eqnarray*}
with $P$ as in (\ref{def/Pnax}). 

In the following we shall prove that 
the factor $e\{b_2 g(n\a+x_1)\}$ can be removed by Fourier analysis, and is hence 
harmless.  Since $g:\Bbb R\to\Bbb R$ is a  
smooth periodic function of period $1$, we have the Fourier expansion 
\begin{eqnarray}\label{eguFouri}
e(b_2 g(u))=\sum_{m\in \Bbb Z} a(m)e(mu),
\end{eqnarray}
where
\begin{eqnarray}\label{am/FoCOEF}
a(m)=\int_0^1 e(b_2 g(u)) e(-mu)du.  
\end{eqnarray}
Note that  $a(m)$ depends on $b_2$ as 
well as the constant $B$ in (\ref{Q>Q<}).   
The series (\ref{eguFouri}) converges absolutely and uniformly in $u\in \Bbb R$, and hence 
\begin{eqnarray}\label{2N/Cont2}
S(N)
&\leq& \bigg|\sum_{n\leq N}\mu(n)e\{b_2 H(n)+P(n)\}
\sum_{m\in \Bbb Z} a(m) e(mx_1+mn\a)\bigg|\nonumber\\
&\leq& \sum_{m\in \Bbb Z} |a(m)|\bigg|\sum_{n\leq N}\mu(n) 
e\{b_2 H(n)+P(n)+mn\a\}\bigg|\nonumber\\ 
&\ll& \sup_{\a, m}\bigg|\sum_{n\leq N}\mu(n) e\{b_2 H(n)+P(n)+mn\a\}\bigg|, 
\end{eqnarray} 
where the implied constant depends on $b_2$ and the constant $B$ in (\ref{Q>Q<}).  
The polynomial $P(n)+mn\a$ is harmless, but the complexity 
comes from $H(n)$ which we deal with in the following subsections. 

\subsection{Theorem~\ref{thm2} with $\a$ irrational.}     
To estimate the right-hand side of (\ref{2N/Cont2}), we rewrite 
the function $H$ in (\ref{Phix=}) as   
\begin{eqnarray*}
H(n)=\sum_{q\in {\mathcal Q}^{\sharp}}F(n;q)
\end{eqnarray*}
where 
\begin{eqnarray}\label{def/Gq=}
F(n;q)=\sum_{q\leq |m|<q^+ \atop q|m} 
\hat{h}(m)e(mx_1) \f{e(nm\a)-1}{e(m\a)-1}. 
\end{eqnarray}
We want to truncate 
$H(n)$ at $Y$, where $Y$ is to be decided a little later.  
Application of (\ref{hhat/UPP}) gives
$$
\hat{h}(m)e(mx_1) \f{e(nm\a)-1}{e(m\a)-1}
\ll e^{-\tau |m|}N, 
$$ 
and therefore 
\begin{eqnarray}\label{trun/Phi}
H(n)
=\sum_{q\in {\mathcal Q}^{\sharp}\atop q\leq Y} F(n;q) 
+O(e^{-\tau Y}N)
=: F(n) +O(e^{-\tau Y}N)  
\end{eqnarray} 
with the implied constants depending on $\tau$. 
If we set  
\begin{eqnarray}\label{DEF/Y=}
Y= \f{8}{\tau}\log N,  
\end{eqnarray}
then the last $O$-term in (\ref{trun/Phi}) is $\ll N^{-7}$,  
and hence (\ref{2N/Cont2}) becomes
\begin{eqnarray}\label{N/Cont2/+}
S(N)
\ll 1 +\sup_{\a, m}|T(N)|, 
\end{eqnarray}
where we should remember the implied constant depends only on $b_2, \tau,$ 
and $B$, 
and where we have written 
\begin{eqnarray}\label{DEF/Tsum}
T(N)
= \sum_{n\leq N}\mu(n) e\{b_2 F(n)+P(n)+mn\a\}. 
\end{eqnarray}
Thus the estimation of $S(N)$ reduces to that of $T(N)$. 

Further analysis on $F(n)$ is necessary. 
Recall that $q^+>q^B$ for any $q\in {\mathcal Q}^{\sharp}. $ Also for 
any $q\in {\mathcal Q}^{\sharp}$,  
we have by (\ref{ratAPP}) that     
\begin{eqnarray*}
\f{|m|}{2qq^+}<\bigg|m\a-\f{ml}{q}\bigg|< \f{|m|}{qq^+}. 
\end{eqnarray*}
If it happens that $q|m$, we change variables $m=qm'$ so that the 
above becomes 
\begin{eqnarray}\label{mqqa<}
\f{|m'|}{2q^+}<\|m'q\a\|< \f{|m'|}{q^+}. 
\end{eqnarray}
For further analysis we write 
\begin{eqnarray*}\label{}
{\mathcal Q}^{\sharp}=\{m_1, m_2, \ldots\}. 
\end{eqnarray*}
Recall that by definition $m_1\geq 2$. Noting that 
\begin{eqnarray}\label{m1m1Bm1+}
m_1<m_1^B\leq m_1^+ 
\leq m_2<m_2^B\leq m_2^+\leq \ldots,   
\end{eqnarray}
we deduce that  
$m_2^+ > m_2^B > (m_1^B)^B = m_1^{B^2}, $ 
and consequently  
\begin{eqnarray}\label{mj+>}
m_j^+ > m_1^{B^j}  
\end{eqnarray}
for all $j\geq 1.$ The sequence (\ref{m1m1Bm1+}) should also 
be truncated at $Y$. Since $m_j\to \infty$ as $j\to\infty$, there exists  
a positive integer $J$ such that    
\begin{eqnarray}\label{mJY<mJ+1}
m_J\leq Y<m_{J+1}. 
\end{eqnarray}
From this and (\ref{mj+>}), we can bound $J$ from above as  
\begin{eqnarray}\label{J<loglog}
J \leq  \f{\log \f{\log Y}{\log m_1}}{\log B}+1 
\ll  \f{\log \log \log N}{\log B} 
\end{eqnarray}
where we have used the definition of $Y$ in (\ref{DEF/Y=}) and therefore 
the implied constant depends on $\tau$. 
If we write $q=m_j$ in (\ref{mqqa<}) 
and change variables as $m=m' m_j, $ 
then 
\begin{eqnarray}\label{Fnmj}
F_j(n):=F(n;m_j)=\sum_{1\leq |m'|< M_j}
\hat{h}(m_j m') e(m_jm' x_1)\f{e(nm' m_j \a)-1}{e(m' m_j\a)-1}  
\end{eqnarray}
where $M_j:=m_j^+/m_j$ for $j=1,\ldots, J-1$, but 
\begin{eqnarray}\label{def/MJ}
M_J:= Y/m_J. 
\end{eqnarray}
In (\ref{Fnmj}) we have 
\begin{eqnarray}\label{ine/m+/}
\f{|m'|}{2 m_j^+}< \|m' m_j\a\|< \f{|m'|}{m_j^+},  
\end{eqnarray}
and if we write $\th_j=\|m_j\a\|$ then the above 
with $m'=1$ gives 
\begin{eqnarray}\label{ine/m'}
\f{1}{2m_j^+} <\th_j < \f{1}{m_j^+} 
\end{eqnarray}
for all $j\geq 1$. Hence (\ref{Fnmj}) can be written as 
\begin{eqnarray}\label{Fjn=fjnth}
F_j(n)= f_j(n\th_j), 
\end{eqnarray}
with 
\begin{eqnarray}\label{Fnmj/C2}
f_j(x) =\sum_{1\leq |m|< M_j}
\hat{h}(m_j m) e(m_jm x_1)\f{e(x m)-1}{e(m \th_j)-1}, \quad x\in [\th_j, \th_j N].  
\end{eqnarray}
We conclude that the function $F(n)$ in (\ref{DEF/Tsum}) 
is of the form 
\begin{eqnarray}\label{F=f1+f2+}
F(n)=f_1(n\th_1)+\cdots f_J(n\th_J). 
\end{eqnarray}
This is the expression from which we start to handle 
the factor $e(b_2 F(n))$ in (\ref{DEF/Tsum}).  

With $f_j$ as in (\ref{Fnmj/C2}) we set   
\begin{eqnarray}\label{||fj||=}
\Phi_j = \sum_{1\leq |m|< M_j} |m|^2 |\hat{h}(m_j m)|.  
\end{eqnarray}
Let $C>0$ be a large constant to be specified in \S6.  
We need to consider three possibilities separately:   
\begin{itemize}
\item[(A)]  $m_J^+ \Phi_J \leq \log^{4C} N; $ 
\item[(B)]  $(m_J^+)^3 \geq \Phi_J N^4 \log^C N$;   
\item[(C)]  $m_J^+ \Phi_J > \log^{4C} N$ and $(m_J^+)^3 < \Phi_J N^4 \log^C N$.     
\end{itemize}

\noindent 
In cases (A) and (B), the factor $e(b_2 F(n))$ will be handled by Fourier analysis 
and Lemma~\ref{lem:VDH}, while 
in case (C) by a finite version of the Vinogradov method 
(Bourgain-Sarnak-Ziegler \cite{BouSarZie}), as well as Poisson summation and 
stationary phase. 

\subsection{Theorem~\ref{thm2} with $\a$ irrational: case (A)}  
In this subsection we prove the following proposition. 

\begin{proposition}\label{prop:caseA} 
Let $S(N)$ be as in (\ref{T2/S=See}), and $h$ an analytic function 
whose Fourier coefficients satisfy the upper bound condition (\ref{hhat/UPP}). 
Assume condition (A). 
Then  
\begin{eqnarray}\label{SN<</fin/A}
S(N)\ll N (\log N)^{8C+5-A},  
\end{eqnarray}
where $A>0$ is arbitrary, and 
the implied constant depends on $A, \tau,$ and $b_2$, but uniform in all the 
other parameters.      
\end{proposition}

We remark that the lower bound condition (\ref{hhat/LOW}) is not needed in 
Proposition~\ref{prop:caseA}.  

\begin{proof} It suffices to bound  
$T(N)$ defined as in (\ref{DEF/Tsum}) under the condition (A).  
Our analysis starts from $f_1$. Recall that 
\begin{eqnarray}\label{Fnmj/2}
f_1(x)=\sum_{1\leq |m| \leq M_1}
\hat{h}(m_1 m) e(m_1 m x_1)\f{e(xm)-1}{e(m \th_1)-1}, 
\qquad x\in [\th_1,\th_1 N]. 
\end{eqnarray}
It is easy to compute the first and second derivatives of $f_1$, that is 
\begin{eqnarray*}
f_1'(x)
&=& 2\pi i   
\sum_{1\leq |m| < M_1} m \hat{h}(m m_1) e(m m_1 x_1) 
\f{e(x m)}{e(m\th_1)-1},  
\qquad x\in [\th_1,\th_1 N],   
\end{eqnarray*} 
and 
\begin{eqnarray*}
f_1''(x)&=& (2\pi i  )^2  
\sum_{1\leq |m| < M_1} m^2 \hat{h}(m m_1) e(m m_1 x_1) 
\f{e(x m)}{e(m\th_1)-1}, 
\qquad x\in [\th_1,\th_1 N]. 
\end{eqnarray*} 
Trivially we have 
\begin{eqnarray*}
|f_1'(x)|\leq  \f{\pi}{2\th_1}\sum_{1\leq |m| < M_1} |\hat{h}(m m_1)|
\leq \f{\pi\Phi_1}{2\th_1}, 
\quad |f_1''(x)|\leq \f{\pi^2 \Phi_1}{\th_1}, 
 \end{eqnarray*} 
where the implied constants are absolute. 
Note that $e(b_2 f_1(x))$ is a smooth periodic function on $\Bbb R$,  
and hence can be expanded into 
Fourier series 
\begin{eqnarray}
e(b_2 f_1(x))=\sum_{k\in \Bbb Z} a(k)e(kx), 
\end{eqnarray}
where 
\begin{eqnarray}
a(k)=\int_0^1 e(b_2 f_1(x))e(-kx)dx. 
\end{eqnarray}
We must compute the dependence of $a(k)$ on $f_1$ and $b_2$.  
By partial integration we have 
\begin{eqnarray*}
a(k)
&=& -\f{1}{2\pi i k}\int_0^1 e(b_2 f_1(x))d e(-kx) \\ 
&=& \f{b_2}{k}\int_0^1 e(b_2 f_1(x))f_1'(x) e(-kx) dx\\ 
&=& \f{b_2}{2\pi i k^2}\int_0^1 \f{d\{e(b_2 f_1(x))f_1'(x)\}}{dx} e(-kx) dx. 
\end{eqnarray*}
Since 
\begin{eqnarray*}
\bigg|\f{d\{e(b_2 f_1(x))f_1'(x)\}}{dx}\bigg|
&=& |e(b_2 f_1(x)) \{f_1''(x)+ 2\pi i b_2 f_1'(x)f_1'(x)\}| \\ 
&\leq& \f{\pi^3}{2} \bigg(\f{\Phi_1}{\th_1}+ b_2\bigg(\f{\Phi_1}{\th_1}\bigg)^2\bigg),  
\end{eqnarray*}
we can bound $a(k)$ as follows 
\begin{eqnarray}\label{ak<<} 
|a(k)| 
\leq \f{\pi^2}{4}\bigg(\f{\Phi_1}{\th_1}+\bigg(\f{\Phi_1}{\th_1}\bigg)^2\bigg)
\f{b_2^2}{|k|^2}
\end{eqnarray}
for $k\not=0$. Obviously for $k=0$ we have $|a(0)|\leq 1$. It follows that 
\begin{eqnarray}\label{sum|ak|}
\sum_{k\in \Bbb Z} |a(k)|
&\leq& 1  + \f{\pi^2}{4} \bigg(\f{\Phi_1}{\th_1}+\bigg(\f{\Phi_1}{\th_1}\bigg)^2\bigg) 
\sum_{|k|\geq 1} \f{b_2^2}{|k|^2} \nonumber\\   
&\leq& (4b_2)^2 \bigg( 1 + \f{\Phi_1}{\th_1}\bigg)^2
\leq (8b_2)^2 (1 + m_1^+ \Phi_1)^2, 
\end{eqnarray}  
where in  the last step we have applied $\sum_{|k|\geq 1}|k|^{-2} < 4$ as well as 
(\ref{ine/m'}).  

Now we can remove the factor $e(b_2 f_1(n \th_1))$ from any sum of the form 
\begin{eqnarray*}
\sum_{n\leq N} \mu(n) e(b_2 f_1(n\th_1)+G(n)) 
\end{eqnarray*} 
where $G(n)$ is a function of $n$. 
Indeed, on inserting the Fourier expansion of $e(b_2 f_1(x))$, 
the above sum in absolute value can be written as 
\begin{eqnarray*}
&=& \bigg|\sum_{n\leq N} \mu(n) e(G(n)) 
\sum_{k_1\in \Bbb Z} a(k_1)e(n k_1 \th_1)\bigg| \\ 
&\leq& \sum_{k_1\in \Bbb Z} |a(k_1)| 
\bigg|\sum_{n\leq N} \mu(n) e(n k_1 \th_1+G(n))\bigg| \\ 
&\leq&
(8b_2)^2 (1 + m_1^+ \Phi_1)^2  
\sup_{k_1, \th_1}\bigg|\sum_{n\leq N} \mu(n) e(n k_1\th_1+G(n))\bigg|,   
\end{eqnarray*} 
by (\ref{sum|ak|}). 
In this way the factor $e(b_2 f_1(n \th_1))$ has been removed. 
Of course, the same argument applies to $e(b_2 f_2), \ldots, e(b_2 f_J)$, 
and hence (\ref{DEF/Tsum}) becomes 
\begin{eqnarray}\label{N/Cont3}
|T(N) | 
\leq \Sigma \Pi,  
\end{eqnarray}
where 
\begin{eqnarray}
\Sigma
= \sup \bigg|\sum_{n\leq N} \mu(n) 
e\{n(k_1 \th_1+\cdots + k_J \th_J)  +P(n) 
+mn\a\}\bigg|
\end{eqnarray}
with the $\sup$ taken over $\a, m, k_1, \ldots, k_J, \th_1, \ldots, \th_J$, 
and where 
\begin{eqnarray}\label{def/PI=}
\Pi= (8b_2)^{2J} \prod_{j=1}^J 
(1 + m_j^+ \Phi_j)^2.  
\end{eqnarray}
The sum $\Sigma$ above can be estimated by Lemma~\ref{lem:VDH}, 
\begin{eqnarray}\label{Sig<<Nlog-AN}
\Sigma \ll N\log^{-A}N, 
\end{eqnarray}
where the implied constant depends on $A$, but independent of 
all the other parameters. 

To estimate $\Pi$ we need to compute $m_1^+ \cdots m_{J-1}^+$. From  (\ref{m1m1Bm1+}) 
we deduce by induction that  
$(m_j^+)^{B^{J-j-1}} \leq m_{J-1}^+$ for $j=1,\ldots, J-1$, 
and therefore 
\begin{eqnarray*}
m_1^+ \cdots m_{J-1}^+ 
\leq (m_{J-1}^+)^{B^{-J+2}+B^{-J+3}+\cdots +B^0}
\leq (m_{J-1}^+)^2. 
\end{eqnarray*}
By definition there is a constant $K\geq 1$ 
depending on $\tau$ such that the inequality $\Phi_j\leq K$ holds for all $j$. Hence 
\begin{eqnarray*}
\prod_{j=1}^{J-1} (1+ m_j^+ \Phi_j)^2 
\leq (2K)^{2(J-1)} (m_1^+ \cdots m_{J-1}^+)^2  
\leq (2K)^{2(J-1)}  (m_{J-1}^+)^4, 
\end{eqnarray*} 
and this can be used to bound $\Pi$ as follows:  
\begin{eqnarray*}\label{OvAll/1}
\Pi 
&=& (8b_2)^{2J} (1+ m_J^+ \Phi_J)^2  
\prod_{j=1}^{J-1} (1+ m_j^+ \Phi_j)^2 \nonumber\\ 
&\leq& (16 b_2 K)^{2J} (1+ m_J^+ \Phi_J)^2 (m_{J-1}^+)^4.  
\end{eqnarray*} 
By (\ref{mj+>}) we have 
\begin{eqnarray*}\label{OvAll/2}
(16 b_2 K)^{2J}=m_1^{\f{2J\log (16 b_2K)}{\log m_1}}  
\leq m_1^{B^{J-1}} \leq m_{J-1}^+ 
\end{eqnarray*}
if $B$ is sufficiently large in terms of $K$ and $b_2$,  
that is in terms of $\tau$ and $b_2$. It turns out that 
for this purpose the choice $B=4[\log (16 b_2K)]$ is acceptable, where 
$[x]$ denotes the integral part of $x$.     
Note that $m_{J-1}^+\leq m_{J}\leq Y$ with $Y$ as in (\ref{DEF/Y=}).  
These together with condition (A) give  
\begin{eqnarray}\label{PI/<=/+}
\Pi 
\leq (1+ m_J^+ \Phi_J)^2 m_J^5 
\leq (1+ m_J^+ \Phi_J)^2 Y^5
\ll  (\log N)^{8C+5}
\end{eqnarray} 
with the implied constant depending on $\tau$ only.    
This is the desired upper bound for $\Pi$. 

Inserting (\ref{PI/<=/+}) and (\ref{Sig<<Nlog-AN}) back into (\ref{N/Cont3}), 
we get 
\begin{eqnarray*}\label{T/fin/A}
T(N) \ll N (\log N)^{8C+5-A},  
\end{eqnarray*}
where the implied constant depends only on $A$ and $\tau$.   
From this and  (\ref{N/Cont2/+}) we conclude that   
\begin{eqnarray*} 
S(N)
\ll 1 +N (\log N)^{8C+5-A}
\end{eqnarray*}
with the implied constant depends on $A, \tau$, and $b_2$, 
but uniform in all the other parameters.   
This completes the analysis of case (A).  
\end{proof} 
 
\section{Theorem~\ref{thm2} with $\a$ irrational: case (B)}  
\setcounter{equation}{0}
In this section we handle case (B). Still, the lower bound 
condition (\ref{hhat/LOW}) is not needed in Proposition~\ref{prop:caseB}.  

\begin{proposition}\label{prop:caseB} 
Let $S(N)$ be as in (\ref{T2/S=See}), and $h$ an analytic function 
whose Fourier coefficients satisfy the upper bound condition (\ref{hhat/UPP}). 
Assume condition (B). Then 
 \begin{eqnarray}\label{S/fin/B}
S(N)\ll N\log^{-A} N + N(\log N)^{6-C}    
\end{eqnarray}
where $A>0$ is arbitrary and the implied constant depends 
on $A, \tau$, and $b_2$ only.  
\end{proposition}

\begin{proof} It is sufficient to estimate 
$T(N)$ defined as in (\ref{DEF/Tsum}) under the condition (B).  
We can repeat the argument in case (A) but with $J$ there 
replaced by $J-1$. Thus in (\ref{DEF/Tsum}) the factors 
$e(b_2 f_1), e(b_2 f_2), \ldots, e(b_2 f_{J-1})$  
can be removed by repeated application of Fourier analysis, 
but the factor $e(b_2 f_{J})$ 
remains in the summation in (\ref{DEF/Tsum}).  
Hence, instead of (\ref{N/Cont3}), we 
have in the present situation,  
\begin{eqnarray}\label{T/S*P*}
|T(N) | 
\leq \Sigma^* \Pi^*  
\end{eqnarray}
with 
\begin{eqnarray}\label{Sig*/B}
\Sigma^* 
= \sup \bigg|\sum_{n\leq N}\mu(n) 
e\{n(k_1 \th_1+\cdots + k_{J-1} \th_{J-1}) + b_2 f_{J}(n\th_J)+P(n)
+mn\a\}\bigg|, 
\end{eqnarray}
where the $\sup$ is taken over $\a, m, k_1, \ldots, k_{J-1}, 
\th_1, \ldots, \th_{J-1}$.   
Also similar to (\ref{def/PI=}),  
\begin{eqnarray}\label{P*/B}
\Pi^*= (8b_2)^{2(J-1)} \prod_{j=1}^{J-1} 
(1+ m_j^+ \Phi_j)^2, 
\end{eqnarray}
where we note that $\Pi^*$ does not have 
any factor involving the subscript $J$. Similar to  
(\ref{PI/<=/+}), we have  
\begin{eqnarray}\label{P*<<L:5}
\Pi^* 
\leq m_J^5 \leq Y^5   
\end{eqnarray}
provided that $B$ is sufficiently large in terms of $\tau$ and $b_2$. 

The estimation of $\Sigma^*$ requires more detailed analysis. We should take 
advantage of the fact that now $\th_J$ is very small. We write $f_J$ in (\ref{Fnmj/C2}) 
in the form 
\begin{eqnarray}\label{Rec/fJnth}
f_J(n\th_J)
=\sum_{1\leq |m|< M_J}
\hat{h}(m_J m) e(m_Jm x_1)
\sum_{j=0}^{n-1} e(jm\th_J),    
\end{eqnarray} 
where recall that $M_J=Y/m_J$ by (\ref{def/MJ}).  
For $j\geq 1$ Taylor's expansion gives 
\begin{eqnarray*}
e(j m \th_J)= \sum_{k=0}^2\f{(2\pi i j m\th_J)^k}{k!}   
+O(j^3 |m|^3 \th_J^3),    
\end{eqnarray*} 
and therefore 
\begin{eqnarray*}
\sum_{j=0}^{n-1} e(j m \th_J)
=1+\sum_{k=0}^2\f{(2\pi i m\th_J)^k}{k!}\sum_{j=1}^{n-1}j^k
+O(|m|^3 \th_J^3N^4). 
\end{eqnarray*}  
Hence  (\ref{Rec/fJnth}) takes the new form  
\begin{eqnarray}\label{fJ=C1C2C3}
f_J(n\th_J)
&=& c_0(M_J) n+ 
\f12 c_1(M_J)\th_J n(n-1) \nonumber\\  
&&+ \f16 c_2 (M_J) \th_J^2 (n-1)n(2n-1) + O\{\widetilde{c}_3(M_J) \th_J^3 N^4\},  
\end{eqnarray}
where 
\begin{eqnarray*}\label{def/C1}
c_k(M) = \f{(2\pi i )^{k}}{k!} \sum_{1\leq |m|< M}m^{k}
\hat{h}(m_J m) e(m_Jm x_1) 
\end{eqnarray*} 
for $k=0, 1, 2$,  while  
\begin{eqnarray*}\label{def/C1}
\widetilde{c}_3(M) = \sum_{1\leq |m|< M} |m|^{3}
|\hat{h}(m_J m)|.   
\end{eqnarray*}  
Obviously $\widetilde{c}_3(M_J)\leq Y\Phi_J$.  Put 
$c_k=c_k(\infty)$ for $k=0, 1, 2$. Then by the upper bound condition (\ref{hhat/UPP}),  
\begin{eqnarray*}\label{def/C1}
|c_k - c_k(M_J)| 
&\leq&  \sum_{|m|\geq M_J}|m|^{k} |\hat{h}(m_J m)| 
\ll \sum_{m\geq M_J} m^{k} e^{-\tau m_J m}  \nonumber\\  
&\ll& e^{-\f12 \tau Y} \ll N^{-4},  
\end{eqnarray*} 
and therefore 
$c_k(M_J)=c_k+O(N^{-4})$ for $k=0, 1, 2$.   
Collecting these estimates back to (\ref{fJ=C1C2C3}), we have  
\begin{eqnarray*}\label{fJ=Qn}
b_2 f_J(n\th_J)= b_2 Q(n) + O(|b_2| N^{-1}) + O(|b_2| Y \Phi_J \th_J^3 N^4)
\end{eqnarray*} 
with 
\begin{eqnarray*}\label{def/Qn=}
Q(n)= c_0 n+ 
\f12 c_1 \th_J n(n-1)+\f16 c_2\th_J^2 (n-1)n(2n-1). 
\end{eqnarray*}
Inserting these back into (\ref{Sig*/B}) yields
\begin{eqnarray*}
\Sigma^* 
&\ll& \sup \bigg|\sum_{n\leq N}\mu(n) 
e\{n(k_1 \th_1+\cdots + k_{J-1} \th_{J-1}) + b_2 Q(n) +P(n)
+mn\a\}\bigg| \\ 
&& +O(|b_2|) +O(|b_2| Y \Phi_J \th_J^3 N^5),   
\end{eqnarray*}
where the $\sup$ is taken over $\a, m, k_1, \ldots, k_{J-1}, 
\th_1, \ldots, \th_{J-1}$.  

The condition (B) is designed to control the last $O$-term, which  
is $\ll N(\log N)^{1-C}$ with the implied constant depending on $\tau$ and $b_2$ only. 
Applying Lemma~\ref{lem:VDH} again to the  above sum over $n$,  we get 
\begin{eqnarray*}
\Sigma^* 
\ll N\log^{-A} N + N(\log N)^{1-C}    
\end{eqnarray*}
where $A>0$ is arbitrary and the implied constant depends 
on $A, \tau$, and $b_2$ only.  
The desired result now follows from this and (\ref{P*<<L:5}).   
\end{proof} 

\section{Theorem~\ref{thm2} with $\a$ irrational: case (C)}  
\setcounter{equation}{0}

\subsection{The result and the idea of proof.} In this section we treat  
case (C) by establishing the following result. 
 
\begin{proposition}\label{prop:caseC} 
Let $S(N)$ be as in (\ref{T2/S=See}), and  $h$ an analytic function whose Fourier 
coefficients satisfying both the upper bound condition (\ref{hhat/UPP}) and the 
lower bound condition  (\ref{hhat/LOW}). Assume condition (C). Then 
\begin{eqnarray}\label{SN=oN/CC}
S(N)=o(N).  
\end{eqnarray}
\end{proposition}
In view of (\ref{N/Cont2/+}), it is sufficient to establish (\ref{SN=oN/CC})  
for $T(N)$ with 
\begin{eqnarray*}
T(N)= \sum_{n\leq N}\mu(n) e\{b_2 F(n)+P(n)+mn\a\}  
\end{eqnarray*}
as in (\ref{DEF/Tsum}).  Here we recall that $P(n)$ 
is the polynomial of degree at most $2$ as in (\ref{def/Pnax}), and   
$F(n)=f_1(n\th_1)+\cdots f_J(n\th_J)$ with  
\begin{eqnarray*}
f_j(x) =\sum_{1\leq |m|< M_j}
\hat{h}(m_j m) e(m_jm x_1)\f{e(x m)-1}{e(m \th_j)-1}, \quad x\in [\th_j, \th_j N]   
\end{eqnarray*}
as in (\ref{F=f1+f2+})  and (\ref{Fnmj/C2}) respectively. 
The tool of our proof is the following result 
of Bourgain-Sarnak-Ziegler \cite{BouSarZie}. 

\begin{lemma}\label{lem:BSZ} 
Let $f: {\Bbb N}\to {\Bbb C}$ with $|f|\leq 1$ and let $\nu$ be a multiplicative
function with $|\nu|\leq 1.$ Let $\tau>0$ be a small parameter 
and assume that for all primes $p_1, p_2\leq e^{1/\tau}, p_1\not=p_2,$  
we have that for $M$ large enough
\begin{eqnarray}
\bigg|\sum _{m\leq M} f(p_1 m)\overline{f(p_2 m)}\bigg|\leq \tau M. 
\end{eqnarray}
Then for $N$ large enough
\begin{eqnarray}
\bigg|\sum _{n\leq N} \nu(n) f(n)\bigg|
\leq 2 \sqrt{\tau\log \f{1}{\tau}}N. 
\end{eqnarray}
\end{lemma} 
Lemma~\ref{lem:BSZ} reduces the estimation of $T(N)$ to that of  
\begin{eqnarray}\label{tTN=d1d2}
\widetilde{T}(N)
=\sum_{n\leq N} e\{b_2 F(d_1 n)- b_2 F(d_2 n)+P(d_1n) - P(d_2n)
+d_1n m \a -d_2n m\a\} 
\end{eqnarray}  
where $d_1\not=d_2$ are positive integers. 
Without loss of generality we assume henceforth that $d_1>d_2$. Noting that 
\begin{eqnarray*}
b_2 F(d_1 n)- b_2 F(d_2 n)
&=& \{b_2 f_1(d_1n\th_1) - b_2 f_1(d_2n\th_1)\}
+\cdots \\ 
&&+ \{b_2 f_{J}(d_1n\th_J) - b_2 f_{J}(d_2n\th_J)\}, 
\end{eqnarray*}
we can repeat the argument in case (A) but with $J$ there 
replaced by $J-1$. Thus in (\ref{tTN=d1d2}) the factors 
$$
e(b_2 f_1), e(- b_2 f_1), \ldots, 
e(b_2 f_{J-1}), e(- b_2 f_{J-1}) 
$$  
can be removed by repeated application of Fourier analysis, 
but the factor 
$$
e\{b_2 f_{J}(d_1n\th_J) - b_2 f_{J}(d_2n\th_J)\}
$$ 
remains in the summation.   
Hence instead of (\ref{N/Cont3}) we 
have in the present situation  
\begin{eqnarray}\label{TN/Cont/4}
|\widetilde{T}(N) | 
\leq \widetilde{\Sigma} \widetilde{\Pi}  
\end{eqnarray}
with new definitions of $\widetilde{\Sigma}$ and $\widetilde{\Pi}$. 
In fact in the above 
\begin{eqnarray}\label{DEF/Sig=/CB}
\widetilde{\Sigma} 
&=& 
\sup \bigg|\sum_{n\leq N}
e\{n(d_1 k_1 \th_1+\cdots + d_1 k_{J-1} \th_{J-1} - 
d_2 l_1 \th_1-\cdots - d_2 l_{J-1} \th_{J-1})  \nonumber\\  
&&  + b_2f_{J}(d_1n\th_J) - b_2 f_{J}(d_2n\th_J)+P(d_1n)-P(d_2n) 
+(d_1-d_2)mn\a\}\bigg|, 
\end{eqnarray}
where the $\sup$ is taken over $\a, m, d_1, d_2, k_1, \ldots, k_{J-1}, 
l_1, \ldots, l_{J-1}, \th_1, \ldots, \th_{J-1}$.   
Also similar to (\ref{def/PI=}),  
\begin{eqnarray*}\label{def/PI=/CB}
\widetilde{\Pi}= (8 b_2)^{4(J-1)} \prod_{j=1}^{J-1} 
(1+ m_j^+ \Phi_j)^4, 
\end{eqnarray*}
where we note that $\widetilde{\Pi}$ does not have 
any factor involving the subscript $J$. 
Similar argument gives  
\begin{eqnarray}\label{TilPi4/I}
\widetilde{\Pi}
\leq m_J^9 \leq Y^9  
\end{eqnarray} 
provided that $B$ is sufficiently large in terms of $\tau$ and $b_2$. 

\medskip 
 
To handle $\widetilde{\Sigma}$, we write $\widetilde{f}_J(x)$ for 
$f_{J}(d_1x)-f_{J}(d_2x)$ so that \begin{eqnarray}\label{Fnmj/C/3}
\widetilde{f}_J(x)
= \sum_{1\leq |m|< M_J}
\hat{h}(m_J m) e(m_J m x_1)\f{e(d_1 m x)-e(d_2 m x)}{e(m \th_J)-1}, 
\quad x\in [\th_J, \th_J N], 
\end{eqnarray}
where recall that $M_J=Y/m_J$ by definition. We want to 
estimate $\widetilde{\Sigma}$ by Poisson's summation formula and 
the method of stationary phase.   
To this end, we need to know the 
derivatives of $\widetilde{f}_J(x)$. 
We are going to use the third derivative of 
$\widetilde{f}_{J}(x)$, which is
\begin{eqnarray}\label{3rd/der}
\widetilde{f}^{(3)}_{J}(x)
= (2\pi i)^3 \sum_{1\leq |m| < M_J} m^3 \hat{h}(m m_J)e(m m_J x_1)
\f{d_1^3 e(d_1m x)- d_2^3  e(d_2 m x)}{e(m\th_J)-1}; 
\end{eqnarray}
the reason for using the third derivative 
will be explained later. 
Since 
$\th_J<\f{1}{m_J^+}$ we have  
$$
|m|\th_J<M_J\th_J \leq \f{1}{m_J} 
$$  
for $|m|<M_J$, and hence 
\begin{eqnarray*}
e(m\th_J)-1=2\pi i m\th_J (1+O(M_J\th_J)). 
\end{eqnarray*} 
It follows that 
\begin{eqnarray}\label{fKJx=}
\widetilde{f}^{(3)}_{J}(x)
= -\f{(2\pi) ^{2}}{\th_J} (\phi^{(3)}(x)+O(d_1^3 M_J\th_J\Phi_J)), 
\end{eqnarray}
where 
\begin{eqnarray}\label{Def/phi=}
\phi^{(3)}(x)=\sum_{1\leq |m| < M_J} m^2\hat{h}(m m_J)e(m m_J x_1)
\{d_1^3 e(d_1mx)- d_2^3  e(d_2mx)\}. 
\end{eqnarray}
The polynomial $\phi^{(3)}(x)$ is too long for a stationary phase argument, 
however the upper and lower bound conditions (\ref{hhat/UPP}) and (\ref{hhat/LOW}) 
enable us to cut $\phi^{(3)}(x)$ at some fixed integer $D$. We will show in the 
following subsection that the choice 
\begin{eqnarray}\label{def/D=[/]}
D=[\tau_2/\tau]+2 
\end{eqnarray}
is acceptable, where $[x]$ denotes the integral part of $x$. 

\subsection{The polynomials $\phi^{(3)}$ and $\phi_D^{(3)}$, and bounds 
for $\widetilde{f}^{(3)}_{J}(x)$}  
We denote 
by $\phi_D^{(3)}$ the part of $\phi^{(3)}$ with $|m|\leq D$, that is   
\begin{eqnarray}\label{DEF/phiD=}
\phi_D^{(3)}(x)=
\sum_{1\leq |m| \leq D} m^2\hat{h}(m m_J)e(m m_J x_1)
\{d_1^3 e(d_1mx)- d_2^3  e(d_2mx)\},  
\end{eqnarray} 
and we want to approximate $\phi^{(3)}$ by this $\phi_D^{(3)}$. 
By the upper bound condition (\ref{hhat/UPP}), the tail 
$\phi^{(3)}-\phi_D^{(3)}$ can be estimated as 
\begin{eqnarray}\label{tail/phi}
\phi^{(3)}(x)-\phi_D^{(3)}(x) 
&\ll& d_1^3 \sum_{m\geq D+1} m^2 |\hat{h}(m m_J)| \nonumber\\ 
&\ll& d_1^3 \sum_{m\geq D+1} m^2 e^{-\tau m m_J} 
\ll d_1^3 e^{-\tau D m_J}, 
\end{eqnarray}
where the implied constants depend at most on $\tau$ and $\tau_2$.  
Next we are going to prove that, when $x$ is away from the zeros of $\phi_D^{(3)}(x)$ 
by a small quantity $\d$, 
$|\phi_D^{(3)}(x)|$ is away from $0$ by some quantity depending on $\d$. 

\begin{lemma}\label{lem: UniAwa}
Let $P(z)$ be a complex polynomial of degree $n$ defined by  
\begin{eqnarray}
P(z)= c_0+ c_1 z + \cdots + c_n z^n,      
\end{eqnarray} 
and let $z_1, \ldots, z_n$ be the zeros of $P(z)$.  Let $\d$ be a small 
real number, and around each $z_j$ make a disc $D_j=\{z: |z-z_j|<\d\}$ 
where $j=1,\ldots, n$. Let ${\Bbb T}$ denote the unit circle. Then for any 
$z\in {\Bbb T}\backslash \{\cup_{j=1}^n D_j\}$ we have  
\begin{eqnarray*}
|P(z)|\geq  \bigg(\f{\d}{3}\bigg)^n \|P\|_2,  
\end{eqnarray*}
where 
\begin{eqnarray}\label{def/||P||}
\|P\|_2= \bigg(\sum_{m=0}^n |c_m|^2\bigg)^{\f12}. 
\end{eqnarray}
\end{lemma}

We remark that ${\Bbb T}\backslash \{\cup_j D_j\}$ is the unit circle with some 
open arcs removed, and some of the removed open arcs may not contain 
any zero of $P(z)$. The total number of these removed open arcs is 
at most $n$.  

\begin{proof} 

Suppose that $|z_j|\leq 2$ for $j=1, \ldots, k$, while 
$|z_j|>2$ for $j=k+1, \ldots, n$. Then we can write 
$P(z)= P_0(z) P_1(z)$ with 
\begin{eqnarray*}
P_0(z)= c_n \prod_{j=1}^k (z-z_j), 
\quad 
P_1(z)=\prod_{j=k+1}^n (z-z_j). 
\end{eqnarray*} 
First we note that a lower bound for $|P(z)|$ follows directly from 
the construction of ${\Bbb T}\backslash \{\cup_{j=1}^n D_j\}$, 
that is 
\begin{eqnarray}\label{Ori/|P|>}
|P(z)|\geq c_n \d^{k} |P_1(z)|,  \quad z\in {\Bbb T}\backslash \{\cup_{j=1}^n D_j\}.  
\end{eqnarray}
Next we compute the norms of $P$ and $P_0$, getting 
\begin{eqnarray*}
\|P_0\|_2^2 
=\int_0^1 |P_0(e(x))|^2 d x 
=\int_0^1 c_n^2 \prod_{j=1}^k |e(x)-z_j|^2 d x
\leq c_n^2 3^{2k}, 
\end{eqnarray*}
and 
\begin{eqnarray*}
\|P\|_2^2
&=&\int_0^1 |P(e(x))|^2 d x 
\leq \max_{z\in \Bbb T}|P_1(z)|^2  
\int_0^1 |P_0(e(x))|^2 d x  \\ 
&\leq& c_n^2 3^{2k} \max_{z\in \Bbb T}|P_1(z)|^2.   
\end{eqnarray*}
The last inequality combined with (\ref{Ori/|P|>}) gives  
\begin{eqnarray}\label{|P|>/Con}
|P(z)|\geq  \bigg(\f{\d}{3}\bigg)^{k} 
\|P\|_2  \f{|P_1(z)|}{\max\limits_{z\in \Bbb T} |P_1(z)|},  
\quad z\in {\Bbb T}\backslash \{\cup_{j=1}^n D_j\}. 
\end{eqnarray}
Suppose $\max\limits_{z\in \Bbb T} |P_1(z)|$ is achieved 
at $z=\zeta\in {\Bbb T}$. Then for any $z\in {\Bbb T}$ 
we have 
\begin{eqnarray*}
\f{|P_1(z)|}{\max\limits_{z\in \Bbb T} |P_1(z)|}
= \prod_{j=k+1}^n \f{|z-z_j|}{|\zeta-z_j|}
\geq \prod_{j=k+1}^n \f{|z_j|-1}{|z_j|+1}
\geq 3^{k-n}.   
\end{eqnarray*}
The desired result finally follows from this and (\ref{|P|>/Con}). 
\end{proof}

We want to apply the above lemma to $\phi_D^{(3)}$. Multiplying  
$\phi_D^{(3)}$ by $e(d_1 Dx)$, we have  
\begin{eqnarray}\label{phi=phi1+phi2}
e(d_1Dx) \phi_D^{(3)}(x)  
=\phi_{D, 1}^{(3)}(x)-\phi_{D, 2}^{(3)}(x), 
\end{eqnarray}  
where, for $\ell =1, 2,$ 
\begin{eqnarray}\label{Def/phiD1}
\phi_{D, \ell}^{(3)}(x)
= d_\ell^3 \sum_{m=-D}^{D} m^2 \hat{h}(m m_J) e(m m_J x_1) e(d_\ell  mx +d_1 Dx).  
\end{eqnarray} 
Recall that we have assumed $d_1>d_2$. The norm of $\phi_{D, \ell}^{(3)}$  
can be computed as 
$$
\|\phi_{D, \ell}^{(3)}\|_2=d_\ell^3 \Phi 
$$
with 
\begin{eqnarray}\label{def/Phi}
\Phi= \bigg(\sum_{m=-D}^{D} |m|^4 |\hat{h}(m m_J)|^2\bigg)^{\f12},  
\end{eqnarray}
and therefore, by (\ref{phi=phi1+phi2}) and the triangle inequality, 
\begin{eqnarray}\label{App/Tri}
\|\phi_D^{(3)}\|_2   
&=&\|\phi_{D, 1}^{(3)}-\phi_{D, 2}^{(3)}\|_2
\geq \|\phi_{D, 1}^{(3)}\|_2-\|\phi_{D, 2}^{(3)}\|_2 \nonumber\\  
&=& (d_1^3-d_2^3)\Phi \geq \Phi. 
\end{eqnarray}  

If we write $z=e(x)$, then $z$ lives on $\Bbb T$ and 
$e(d_1Dx) \phi_D^{(3)}(x)$ can be written as a polynomial,  
say $P(z)$, in $z$ with degree $2 d_1 D$. An application of Lemma~\ref{lem: UniAwa} 
to $P(z)$ asserts that    
\begin{eqnarray}\label{App/LemAwa}
|P(z)|  
\geq \bigg(\f{\d}{3}\bigg)^{2 d_1 D} \|P\|_2, 
\quad z\in {\Bbb T}\backslash \{\cup_{j=1}^n D_j\},  
\end{eqnarray} 
where $\|P\|_2$ is defined as in (\ref{def/||P||}). Obviously $\|P\|_2=\|\phi_D^{(3)}\|_2$. 
   
Under the map $x\mapsto z=e(x)$, the pre-image of 
$z\in {\Bbb T}\cap \{\cup_{j=1}^n D_j\}$  
is a union of small intervals 
$$
\bigcup_{\ell\leq L} I_\ell \subset (0, 1], 
$$ 
where $L\leq \deg (P) = 2d_1D$.  Note that each $I_\ell$ has length at most 
$2\d$. It follows from (\ref{App/Tri}) and 
(\ref{App/LemAwa}) that, 
for $x\in (0, 1]\backslash \{\cup_{\ell\leq L} I_\ell\}$, 
\begin{eqnarray*}
|\phi_D^{(3)}(x)|
\geq \bigg(\f{\d}{3}\bigg)^{2 d_1 D} \Phi.  
\end{eqnarray*} 
Obviously $\Phi\geq |\hat{h}(m_J) |$, which together with (\ref{tail/phi})  gives 
\begin{eqnarray}\label{12phiD-Tail}
\f12 |\phi_D^{(3)}(x)|  - |\phi^{(3)}(x)-\phi_D^{(3)}(x)|  
\geq \f12 \bigg(\f{\d}{3}\bigg)^{2 d_1 D} |\hat{h}(m_J)| 
- K d_1^3 e^{-\tau D m_J},   
\end{eqnarray}  
where $K=K(\tau, \tau_2)$ is the final constant implied in (\ref{tail/phi}).   
The lower bound condition (\ref{hhat/LOW}) implies that 
$|\hat{h}(m_J) | \gg e^{-\tau_2 m_J}$, and hence the right-hand side of 
(\ref{12phiD-Tail}) is positive provided that $m_J$ is large and 
\begin{eqnarray}\label{CON/d1}
d_1^3 \leq \bigg(\f{\d}{3}\bigg)^{2d_1D}
\f{e^{(\tau D - \tau_2) m_J}}{2Km_J}. 
\end{eqnarray} 
In view of (\ref{def/D=[/]}) and (\ref{mJY<mJ+1}),  
the exponent $(\tau D - \tau_2) m_J$ approaches infinity when $N\to\infty$. 
Suppose that (\ref{CON/d1}) is satisfied. 
Then, for $x\in (0, 1]\backslash \{\cup_{\ell\leq L} I_\ell\}$, 
\begin{eqnarray}\label{Fin/Lphi}
|\phi^{(3)}(x)|  
\geq \f12 |\phi_D^{(3)}(x)| + \bigg(\f12 |\phi_D^{(3)}(x)|
-|\phi^{(3)}(x)-\phi_D^{(3)}(x)|\bigg)   
\geq \f12 \bigg(\f{\d}{3}\bigg)^{2 d_1 D} \Phi.  
\end{eqnarray}  
We collect the above analysis to get the following result.  

\begin{lemma}\label{zero/ULbd} 
Let notations be as above and assume (\ref{CON/d1}). If 
\begin{eqnarray}\label{Con/d/an}
d_1^3 \leq \bigg(\f{\d}{3}\bigg)^{2d_1D}\f{1}{\th_JY^3}, 
\end{eqnarray} 
then, for $x\in (0, 1]\backslash \{\cup_{\ell\leq L} I_\ell\}$, 
\begin{eqnarray*}\label{tilfJ3/Lb}
|\widetilde{f}_J^{(3)}(x)|
\gg \f{\Phi_J}{\th_JY} \bigg(\f{\d}{3}\bigg)^{2 d_1 D}, 
\end{eqnarray*}
where the implied constant depends at most on $\tau$ and $\tau_2$.   
\end{lemma} 

At the present stage we do not need to know 
which one of (\ref{Con/d/an}) and (\ref{CON/d1}) 
is more restrictive. From now on we assume both (\ref{Con/d/an}) and (\ref{CON/d1}), 
and in \S6.3 we will show that they are both satisfied by choosing $\d$ and $C$ properly.  

\begin{proof} To prove the lemma we must compare $\Phi$  with $\Phi_J$.   
The definitions (\ref{def/Phi})  and (\ref{||fj||=}) trivially imply 
\begin{eqnarray*}
\Phi^2\leq \sum_{1\leq |m|< M_J} |m|^4 |\hat{h}(m_J m)|^2 \leq \Phi_J^2. 
\end{eqnarray*}
In the other direction we have by Cauchy's inequality that 
\begin{eqnarray*}
\Phi_J^2 
\leq 2M_J \sum_{1\leq |m|< M_J} |m|^4 |\hat{h}(m_J m)|^2. 
\end{eqnarray*}
We cut the last sum at $D$; by the argument in (\ref{tail/phi}) and the upper bound condition (\ref{hhat/UPP}), the tail can be estimated as  
\begin{eqnarray*}
\sum_{D+1 \leq |m|< M_J}
|m|^4 |\hat{h}(m_J m)|^2
\ll e^{-2\tau D m_J}, 
\end{eqnarray*} 
where the implied constant depends at most on $\tau$ and $\tau_2$.  
The last quantity is $\ll e^{-2\tau_2 m_J}\ll |\hat{h}(m_J)|^2\leq \Phi^2$  
by the definition of $D$ in (\ref{def/D=[/]}) 
as well as the lower bound condition (\ref{hhat/LOW}). It follows that 
\begin{eqnarray*}
\sum_{1\leq |m|< M_J} |m|^4 |\hat{h}(m_J m)|^2 
\ll \Phi^2, 
\end{eqnarray*}
that is $\Phi_J^2\ll M_J \Phi^2$, 
where the implied constant depends at most on $\tau$ and $\tau_2$.  

We deduce form this and (\ref{Fin/Lphi}) that, 
for $x\in (0, 1]\backslash \{\cup_{\ell\leq L} I_\ell\}$, 
\begin{eqnarray*}
|\phi^{(3)}(x)|\gg  \bigg(\f{\d}{3}\bigg)^{2 d_1 D}\f{\Phi_J}{M_J}, 
\end{eqnarray*}
and hence (\ref{Con/d/an}) and (\ref{fKJx=}) imply  
\begin{eqnarray*} 
|\widetilde{f}^{(3)}_{J}(x)| 
\gg \f{\Phi_J}{\th_JM_J} \bigg(\f{\d}{3}\bigg)^{2 d_1 D} 
\end{eqnarray*}
where the implied constants depend at most on $\tau$ and $\tau_2$.  
The desired result now follows from this and $M_J\leq Y$. \end{proof} 

In applications we must reformulate Lemma~\ref{zero/ULbd} for   
the function $\widetilde{f}_J^{(3)}(x\th_J)$ with $x\in (0, \th_J^{-1}].$  
Write $x\th_J=\xi$ and   
\begin{eqnarray}\label{def/Jell}
J_\ell = \th_J^{-1} I_{\ell}, 
\end{eqnarray} 
that is each $J_\ell$ is an amplification of $I_\ell$ by $\th_J^{-1}$. Note that  
the length of each $J_\ell$ is $\leq 2\th_J^{-1}\d$. 
Hence Lemma~\ref{zero/ULbd} implies that,  
for $x\in (0, \th_J^{-1}]\backslash \{\cup_{\ell\leq L} J_\ell\}$,    
\begin{eqnarray}\label{Lphi/xi}
|\widetilde{f}_J^{(3)}(x)|
\gg \f{\Phi_J}{\th_JY} \bigg(\f{\d}{3}\bigg)^{2 d_1 D}. 
\end{eqnarray} 
On the other hand we deduce trivially from (\ref{3rd/der}) that, for all real $x$, 
\begin{eqnarray}\label{Upp/phixth} 
|\widetilde{f}_J^{(3)}(x\th_J)|\ll d_1^3\f{\Phi_J}{\th_J}. 
\end{eqnarray}
The implied constants in (\ref{Lphi/xi}) and (\ref{Upp/phixth}) are absolute. 
These bounds will be used in the following subsection. 

\subsection{Application of Poisson's summation and stationary phase.}   
In this subsection we estimate $\widetilde{\Sigma}$ in (\ref{DEF/Sig=/CB}) by 
Poisson's summation formula and stationary phase. The following lemma of van der Corput 
(see for example Iwaniec and Kowalski \cite{IwaKow}, Theorem~8.20),    
in particular, will be applied. 

\begin{lemma}\label{lem:vdCorput}
Let $b-a\geq 1.$ Let $F(x)$ be a real function on $(a, b)$ and $k\geq 2$ 
such that 
\begin{eqnarray}
\L\leq |F^{(k)}(x)|\leq \eta\L
\end{eqnarray}
for some $\L>0$ and $\eta\geq 1$. Then
\begin{eqnarray*}
\sum_{a<n<b}e(F(n))\ll \eta^{2^{2-k}}\L^{\kappa}(b-a)+\L^{-\kappa}(b-a)^{2^{2-k}}, 
\end{eqnarray*}
where $\kappa=(2^k-2)^{-1}$ and the implied constant is absolute. 
\end{lemma}

The sum $\widetilde{\Sigma}$ in (\ref{DEF/Sig=/CB}) can be written as 
\begin{eqnarray}\label{DEF/S/C}
\widetilde{\Sigma}
= \sup \bigg|\sum_{n\leq N}e(E(n))\bigg| 
\end{eqnarray}
with 
\begin{eqnarray*}
E(x)
&=&  x(d_1 k_1 \th_1+\cdots + d_1 k_{J-1} \th_{J-1} - 
d_2 l_1 \th_1-\cdots - d_2 l_{J-1} \th_{J-1})  \nonumber\\  
&&  + b_2 \widetilde{f}_J(x\th_J)+P(d_1x)-P(d_2x) 
+(d_1-d_2)mx\a, 
\end{eqnarray*}
where the $\sup$ is taken over $\a, m, d_1, d_2, k_1, \ldots, k_{J-1}, 
l_1, \ldots, l_{J-1}, \th_1, \ldots, \th_{J-1}$. 
If we take the third derivative of $E(x)$, then all the quadratic and linear terms 
in $E(x)$ will be killed, and the argument will be clearer. This is the reason for taking the 
third derivative of $E(x)$. Thus (\ref{def/Pnax}) implies that 
\begin{eqnarray}\label{E(3)x=f}
E^{(3)}(x)
=  b_2 \widetilde{f}^{(3)}_J(x\th_J)\th_J^3. 
\end{eqnarray}
Recall that in case (C) we have $m_J^+ \Phi_J > \log^{4C}N$.  
We need to handle the following two possibilities separately:    
\begin{itemize}
\item[(C1)]  $m_J^+ \leq N$; 
\item[(C2)]  $m_J^+> N$.  
\end{itemize}

\begin{proof}[Proof of Proposition~\ref{prop:caseC} under (C1).]   
In this case we will first conduct our analysis 
on the subinterval $(0, \th_J^{-1}]\subset (0, N]$.  
The set $(0, \th_J^{-1}]\backslash \{\cup_{\ell\leq L} J_\ell\}$ consists of  
at most $L+1$ intervals, and we suppose $(a, b)$ is any one of them. 
On this interval $(a, b)$  we apply (\ref{Lphi/xi}) and (\ref{Upp/phixth}) to get 
\begin{eqnarray}\label{E(3)UpLo}
\b \th_J^2\Phi_J 
\ll 
|E^{(3)}(x)|
\ll d_1^3 \th_J^2\Phi_J 
\end{eqnarray}
with 
\begin{eqnarray}\label{b=d/32dD}
\b=\bigg(\f{\d}{3}\bigg)^{2 d_1 D}\f{1}{Y},  
\end{eqnarray} 
where the implied constants depend on $b_2, \tau,$ and $\tau_2$ only.  
This means that we can take 
$\L =\b \th_J^2\Phi_J$ and $\eta =  \b^{-1} d_1^3$  
in Lemma~\ref{lem:vdCorput} with $k=3$, 
which implies that 
\begin{eqnarray*}
\sum_{n\in (a, b)}e(E(n)) 
\ll \b^{-\f13} d_1^2 (\th_J^2\Phi_J)^{\f16}(b-a)+ (\b \th_J^2\Phi_J)^{-\f16} (b-a)^{\f12} + 1, 
\end{eqnarray*}
where we have added a $1$ on the right-hand side to cover the case $b-a<1$. 
Summing over all these possible intervals 
$(a, b)\subset (0, \th_J^{-1}]\backslash \{\cup_{\ell\leq L} J_\ell\}$, which are at most 
$L+1\leq 2d_1D+1$ in number, we get 
\begin{eqnarray}\label{Sum/E/Punc}
\sum_{n\in (0, \th_J^{-1}]\backslash \{\cup J_\ell\}} e(E(n)) 
\ll \b^{-\f13}d_1^3 (\th_J^2\Phi_J)^{\f16}\th_J^{-1}
+ d_1 (\b \th_J^2\Phi_J)^{-\f16}\th_J^{-\f12} 
+ d_1,  
\end{eqnarray}
where the implied constants depend on $b_2, \tau,$ and $\tau_2$ only.  
The length of each interval $J_\ell$ is $\ll \th_J^{-1}\d$ by (\ref{def/Jell}), 
and hence trivially  
\begin{eqnarray*}
\sum_{n\in J_\ell} e(E(n))\ll \th_J^{-1}\d.   
\end{eqnarray*}
The number $L$ of these intervals $J_\ell$ is at most $2d_1 D$, and consequently 
\begin{eqnarray*}
\sum_{n\in \cup J_\ell} e(E(n)) 
\ll  d_1 \th_J^{-1}\d, 
\end{eqnarray*} 
which together with (\ref{Sum/E/Punc}) yields  
\begin{eqnarray}\label{Sum/0thJ}
\sum_{n\in (0, \th_J^{-1}]}e(E(n)) 
\ll \b^{-\f13}d_1^3(\th_J^2\Phi_J)^{\f16}\th_J^{-1}
+ d_1 (\b \th_J^2\Phi_J)^{-\f16}\th_J^{-\f12} 
+ d_1+ d_1\th_J^{-1}\d   
\end{eqnarray}
where the implied constants depend on $b_2, \tau,$ and $\tau_2$ only.  

Now we come to the estimation of $\widetilde{\Sigma}$. 
We cut the interval $(0, N]$ into two smaller ones 
$(0, \th_J^{-1}K] \cup (\th_J^{-1}K, N]$, where $K=[\th_JN]$ and 
$[x]$ means the integral part of $x$. We may assume the second interval 
$(\th_J^{-1}K, N]$ has length $\geq (2\th_J)^{-1}$,   since otherwise 
we use $(0, N]= (0, \th_J^{-1}(K+1)] \backslash (N, \th_J^{-1}(K+1)]$. 
Hence we split  
(\ref{DEF/S/C}) as 
\begin{eqnarray}\label{D/S/SS} 
\widetilde{\Sigma}
\leq \widetilde{\Sigma}_0+ \widetilde{\Sigma}_1
\end{eqnarray} 
where 
\begin{eqnarray*} 
\widetilde{\Sigma}_0=\sup \bigg|\sum_{n\in (0, \th_J^{-1}K]} e(E(n))\bigg|, \quad 
\widetilde{\Sigma}_1=\sup \bigg|\sum_{n\in (\th_J^{-1}K, N]} e(E(n))\bigg| 
\end{eqnarray*} 
with the sup having the same meaning as in (\ref{DEF/S/C}). 
The last sum $\widetilde{\Sigma}_1$ can be bounded from above by the 
right-hand side of (\ref{Sum/0thJ}). 
The main interval $(0, \th_J^{-1}K]$ is the union of  
$K$ smaller intervals $L_k: =(\th_J^{-1}(k-1), \th_J^{-1}k]$ with $k=1, \ldots, K$, 
and (\ref{Sum/0thJ}) holds with the interval $(0,\th_J^{-1}]$ therein 
replaced by any $L_k$ with $k=2, 3, \ldots, K.$ 
It follows that  
\begin{eqnarray}\label{wtSllMes}
\widetilde{\Sigma}_0
\ll \b^{-\f13}d_1^3 (\th_J^2\Phi_J)^{\f16}N
+ d_1 (\b \th_J^2\Phi_J)^{-\f16}\th_J^{\f12} N 
+ d_1\th_J N+ d_1\d N. 
\end{eqnarray}
The first term on the right-hand 
side of  (\ref{wtSllMes}) is bounded from above by $\ll d_1^3 \b^{-\f13}\th_J^{\f13}N$ 
with the implied constant depending on $\tau$ only,    
and the second by $d_1 \b^{-\f16} \th_J^{\f16}  \Phi_J^{-\f16} N$, 
which also dominates the third term. Therefore the third term can be erased, and  
consequently (\ref{D/S/SS}) becomes 
\begin{eqnarray}\label{wtSll/fin}
\widetilde{\Sigma}
\ll d_1^3 \b^{-\f13} \f{N}{(m_J^+)^{\f13}}  
+ d_1 \b^{-\f16}  \f{NY}{(m_J^+ \Phi_J)^{\f16}}
+ d_1 \d N 
\end{eqnarray}
where the implied constant depends on $b_2, \tau, \tau_2$ only.  
We multiply $\widetilde{\Pi}$ with $\widetilde{\Sigma}$,   
and then apply the bound (\ref{TilPi4/I}) to get 
\begin{eqnarray*}\label{tilT/C/Ag}
|\widetilde{T}(N)| 
\leq \widetilde{\Sigma} \widetilde{\Pi} 
\ll d_1^3 \b^{-\f13} \f{N Y^9}{(m_J^+)^{\f13}}
+ d_1 \b^{-\f16} \f{N Y^{10}}{(m_J^+ \Phi_J)^{\f16}}
+ d_1 m_J^{9}\d N, 
\end{eqnarray*}
where the implied constant depends on $b_2, \tau, \tau_2$ only. It should be 
remarked that to the last term on the right-hand side above, the bound  
$\widetilde{\Pi}\leq m_J^9$ has been used instead of the crude bound 
$\widetilde{\Pi}\leq Y^9$. 

Now we specify 
\begin{eqnarray}\label{del=}
\d=3 m_J^{-10}, \quad C\geq 20 d_1 D+20, 
\end{eqnarray}
so that (\ref{b=d/32dD}) implies that 
\begin{eqnarray}\label{beta/d=}
\b^{-1}= m_J^{20 d_1 D}Y
\leq Y^{20d_1D+1},  
\end{eqnarray}
and hence 
\begin{eqnarray}\label{tilT/C1/fin}
\widetilde{T}(N)  
\ll \f{d_1^3 N Y^{7d_1D+10}}{(m_J^+)^{\f13}}
+ \f{d_1 N Y^{4d_1D+11}}{(m_J^+ \Phi_J)^{\f16}} 
+ \f{d_1 N}{m_J},  
\end{eqnarray}
where the implied constant depends on $b_2, \tau, \tau_2$ only. 
Applying the assumption $m_J^+\gg m_J^+\Phi_J\geq \log^{4C} N$ we get 
\begin{eqnarray}\label{S0Pi<<} 
\widetilde{T}(N) =o(N)
\end{eqnarray}
as $N\to \infty.$ 

We must check that our choices of $\d$ and $C$ in (\ref{del=}) make 
the inequalities (\ref{CON/d1}) and (\ref{Con/d/an}) meaningful, 
that is neither (\ref{CON/d1}) nor (\ref{Con/d/an})  
confines $d_1$ to a finite interval. This can be seen from 
the fact that under (\ref{del=})  
the right-hand side of (\ref{CON/d1}) equals 
$$
\f{e^{(\tau D - \tau_2) m_J}}{2m_J^{20d_1D+1}}
$$  
which clearly approaches infinity as $m_J\to\infty$, that is as $N\to \infty$. 
Also under (\ref{del=})  
the right-hand side of (\ref{Con/d/an}) is, 
by the assumption $m_J^+\gg m_J^+\Phi_J\geq \log^{4C} N$ again,    
\begin{eqnarray*}
\gg \f{m_J^+}{Y^3m_J^{20d_1D}}
\gg \f{\log^{4C}N}{Y^{20d_1D+3}}
\end{eqnarray*} 
which also approaches to infinity as $N\to \infty$. Thus our choices of $\d$ and $C$ 
are indeed acceptable.  
This completes our analysis concerning  
the sum $\widetilde{T}(N)$.    
The desired result for $S(N)$ follows from (\ref{N/Cont2/+}), and this finishes the analysis in 
case (C1). 
\end{proof} 

\subsection{The case (C2).} 
In the present situation we start the analysis on $(0, N]$ directly, instead of on $(0, \th_J^{-1}]$.  
First we remark that the case $m_J^+ \sqrt{\d}\leq N$ can be treated in the 
same way  as in (\ref{Sum/0thJ}), where $\d$ is the same as in 
the proof for (C1) and has been specified  
as $\d=3 m_J^{-10}$  in  (\ref{del=}).  The reason is that 
$\d=o(\sqrt{\d})$ as $\d\to 0$. Hence from now on we assume $m_J^+ \sqrt{\d}>N$.  

\begin{lemma}\label{lem:LA} 
Let $\l=(\l_{-D}, \ldots, \l_{-1}, \l_1, \ldots, \l_D)\in {\Bbb C}^{2D}$ 
be a non-zero vector. Then there are two  
constants $\ve_0$ and $4\leq k\leq 2D+3$, both of which depend on $D$ only, such that 
\begin{eqnarray}
\bigg|\sum_{1\leq |m|\leq D} \l_m m^{k} \bigg|
> \ve_0 \|\l\|_2,    
\end{eqnarray}
where  $\|\l\|_2$ denotes the $l^2$ norm of $\l$.  
\end{lemma} 

\begin{proof} The vectors $(m^4, m^5, \ldots, m^{2D+3})$ with $m=\pm 1, \ldots, \pm D$ 
form a $(2D)\times (2D)$ Vandermonde matrix  whose determinant is non-zero, and therefore the vector $\l$ 
cannot be perpendicular to all of these vectors.  
\end{proof}

\begin{proof}[Proof of Proposition~\ref{prop:caseC} under (C2).] 
Similar to (\ref{fKJx=}) and (\ref{Def/phi=}) we have, for any positive 
integer $k$ and $x\leq \th_J N\ll \sqrt{\d}$,  
\begin{eqnarray}\label{fKJx=/k}
\widetilde{f}^{(k)}_{J}(x)
= \f{(2\pi i)^{k-1}}{\th_J} (\phi^{(k)}(x)+O(d_1^k M_J^{k-2} \th_J\Phi_J)), 
\end{eqnarray}
where 
\begin{eqnarray}\label{Def/phi=/k}
\phi^{(k)}(x)=\sum_{1\leq |m| < M_J} m^{k-1}\hat{h}(m m_J)e(m m_J x_1)
\{d_1^k e(d_1mx)- d_2^k  e(d_2mx)\},  
\end{eqnarray} 
and we recall  that $\Phi_J$ is defined as in (\ref{||fj||=}).  
Let $\phi_D^{(k)}(x)$ be the sum in (\ref{Def/phi=/k}) 
over the shorter range  $1\leq |m|\leq D$ with $D$ as in 
(\ref{def/D=[/]}). By the upper 
bound condition (\ref{hhat/UPP}), the tail 
$\phi^{(k)}-\phi_D^{(k)}$ can be estimated as 
\begin{eqnarray}\label{tail/phi/k}
\phi^{(k)}(x)-\phi_D^{(k)}(x) 
\ll d_1^k \sum_{m\geq D+1} m^{k-1} |\hat{h}(m m_J)|
\ll d_1^k e^{-\tau D m_J}, 
\end{eqnarray}
where the implied constants depend at most on $\tau$ and $\tau_2$.  
Since $x$ is small, we 
can expand each exponential function in $\phi_D^{(k)}(x)$, 
getting 
\begin{eqnarray*}
\phi_D^{(k)}(x)
&=& (d_1^{k}-d_2^{k}) 
\sum_{1\leq |m|\leq D} m^{k-1}\hat{h}(m m_J)e(m m_J x_1) \\ 
&& +O\{(d_1^{k+1}-d_2^{k+1})  x \Phi \} 
\end{eqnarray*} 
with $\Phi$ as in (\ref{def/Phi}).  
Fix a $k\geq 4$ such that Lemma~\ref{lem:LA} holds, that is for this 
$k$ and some positive $\ve_0$ depending on $D$ we have 
\begin{eqnarray*}
\bigg|\sum_{1\leq |m|\leq D} m^{k-1}\hat{h}(m m_J)e(m m_J x_1)\bigg| 
> \ve_0 \Phi.  
\end{eqnarray*}
Hence, for this $k$ and $x\leq \th_J N$,  
\begin{eqnarray*}
|\phi_D^{(k)}(x)| 
\gg d_1^{k-1} (1- d_1^2 x) \Phi,  
\end{eqnarray*}
where the assumption $d_1>d_2$ has been applied. If  
\begin{eqnarray}\label{C2/d1<}
d_1< (16\d)^{-\f14}  
\end{eqnarray}
then for $x\leq \th_J N$ we have, as in the proof of Lemma~\ref{zero/ULbd},   
$$
|\phi_D^{(k)}(x)| \gg d_1^{k-1} \Phi \gg d_1^{k-1} \f{\Phi_J}{Y},  
$$  
from which and (\ref{tail/phi/k}) it follows 
that $|\phi^{(k)}(x)| \gg d_1^{k-1} \Phi_J Y^{-1}.$  
Inserting this and the trivial upper bound 
$\phi^{(k)}(x)\ll d_1^k \Phi_J Y^{k-3}$ into (\ref{fKJx=/k}) yields  
that, for all real $x\leq N$, 
\begin{eqnarray*}\label{Up/Lo/k} 
d_1^{k-1}\f{\Phi_J}{\th_J Y} 
\ll |\widetilde{f}_J^{(k)}(x\th_J)|\ll d_1^{k}\f{\Phi_J Y^{k-3}}{\th_J},  
\end{eqnarray*}
which together with 
$E^{(k)}(x)= b_2 \widetilde{f}^{(k)}_J(x\th_J)\th_J^k$ gives  
\begin{eqnarray}\label{<<Ekx<<} 
d_1^{k-1}\f{\th_J^{k-1}\Phi_J}{Y} \ll |E^{(k)}(x)|\ll d_1^{k} \th_J^{k-1}\Phi_J Y^{k-3}.   
\end{eqnarray}
This corresponds to (\ref{E(3)UpLo}) in case (C1), 
and it means in Lemma~\ref{lem:vdCorput} we can take 
\begin{eqnarray}
\L= d_1^{k-1}\f{\th_J^{k-1}\Phi_J}{Y}, \quad 
\eta=d_1 Y^{k}.  
\end{eqnarray}
It follows that 
\begin{eqnarray*}
\sum_{n\leq N}e(E(n))
&\ll& (d_1 Y^k)^{2^{2-k}}\bigg(d_1^{k-1}\f{\th_J^{k-1}\Phi_J}{Y}\bigg)^{\kappa}N
+\bigg(d_1^{k-1}\f{\th_J^{k-1}\Phi_J}{Y}\bigg)^{-\kappa}N^{2^{2-k}} \\ 
&\ll& d_1 Y^2 \th_J^{k\kappa-\kappa} N
+ Y \th_J^{-k\kappa} \f{N^{2^{2-k}}}{(m_J^+ \Phi_J)^{\kappa}} \\ 
\end{eqnarray*}
where $\kappa=(2^k-2)^{-1}$ and the implied constant depends on 
$b_2, \tau, $ and $\tau_2$ only. Multipling by 
$\widetilde\Pi$ and applying (\ref{TilPi4/I}),  we have 
\begin{eqnarray}\label{C2/TN<<}
|\widetilde{T}(N)| 
&\leq& \widetilde{\Sigma}\widetilde{\Pi} 
\ll d_1 Y^{11} \th_J^{k\kappa-\kappa} N
+ Y^{10} \th_J^{-k\kappa} \f{N^{2^{2-k}}}{(m_J^+ \Phi_J)^{\kappa}} \nonumber\\  
&\ll& d_1 Y^{11} \f{N}{(m_J^+)^{k\kappa-\kappa}}
+ Y^{10} (m_J^+)^{k\kappa}  N^{2^{2-k}}.  
\end{eqnarray}
Here we have simply erased the denominator of the second term since it 
is bigger than $1$. To the term $(m_J^+)^{k\kappa-\kappa}$ in the denominator we apply  
the assumption $m_J^+ > N$ in case (C2),  
while to 
the term $(m_J^+)^{k\kappa}$ in the numerator we use the assumption 
$(m_J^+)^3 < N^4 \log^C N$ in case (C). 
The exponent of $N$ in the last term of  
(\ref{C2/TN<<}) is 
$$
\f43 k\kappa + 2^{2-k}=\f{4k}{3(2^k-2)}+2^{2-k} <1  
$$ 
when $k\geq 4$. 
This proves that $\widetilde{T}(N)=o(N)$ 
and hence $T(N)=o(N)$ by Lemma~\ref{lem:BSZ} again.    
This completes the analysis in case (C2), and 
Proposition~\ref{prop:caseC} is finally proved.  
\end{proof} 

\begin{proof}[Proof of Theorem~\ref{thm2}.] 
Theorem~\ref{thm2} follows from 
Propositions~\ref{prop/rat}, 
\ref{prop:caseA}, \ref{prop:caseB}, and 
\ref{prop:caseC}.  
\end{proof} 

\section{Disjointness of $\mu$ from Furstenberg's system}  
\setcounter{equation}{0}
\subsection{Furstenberg's example.} 

Furstenberg gave an example of smooth transformation $T: {\Bbb T}^2\to {\Bbb T}^2$ 
such that the ergodic averages do not all exist. 
Let $\a$ be as in \S 4.1 such that 
\begin{eqnarray}\label{qk+1>eqk}
q_{k+1}\asymp e^{\tau q_k}  
\end{eqnarray} 
with $\tau$ as in (\ref{hhat/UPP}). 
Define $q_{-k}=q_k$ and set   
\begin{eqnarray}\label{Furs/h=}
h(x)=\sum_{k\not=0}\f{e(q_k\a)-1}{|k|} e(q_k x). 
\end{eqnarray}
It follows from (\ref{ratAPP}) and (\ref{qk+1>eqk}) that 
$h(x)$ is a smooth function. We also have 
$h(x)=g(x+\a)-g(x)$ 
where 
\begin{eqnarray}
g(x)=\sum_{k\not=0}\f{1}{|k|}e(q_k x)
\end{eqnarray}
so that $g(x)\in L^2(0,1)$ and in particular   
defines and measurable function. But $g(x)$ cannot 
correspond to a continuous function, as shown in Furstenberg \cite{Fur61}. 
Hence $e(\l g(x))$ is not continuous for suitable $\l$ and according to 
Furstenberg it follows that $T_{\l}$ defined by (\ref{def/SKEW}) with 
$h_{\l}=\l h$ is irregular. 

In the next subsection we deduce directly from Theorem~\ref{thm2} 
that $T_{\l}$ is linearly disjoint from $\mu$.  For this specific skew product 
one can prove the disjointness of $T_{\l}$ with $\mu$ even with a rate 
as shown in \cite{LiuSar}.

\subsection{The M\"obius function is disjoint from the Furstenberg example.}   
It is enough to prove that a smooth conjugation of  
Furstenberg's dynamical system above satisfies the conditions of Theorem~\ref{thm2}. 
To this end we introduce another function 
\begin{eqnarray}
H(x)=\sum_{m\in \Bbb Z} \hat{H}(m) e(m x), 
\end{eqnarray}
where 
\begin{eqnarray}\label{hatH=}
\hat{H}(m)=e^{-2\tau |m|}. 
\end{eqnarray}
Obviously $H(x)=G(x+\a)-G(x)$ 
where 
\begin{eqnarray}\label{def/G=}
G(x)=\sum_{m\in \Bbb Z} \hat{H}(m)\f{e(mx)}{e(m\a)-1}.  
\end{eqnarray} 
We claim that 
$G(x)$ is smooth, and this can be proved by the 
the argument in Lemma~\ref{lem:2ser}. In fact by (\ref{posh/t})  
for any positive $t$, 
\begin{eqnarray*}
\sum_{m\leq t\atop q_k\nmid m} \f{1}{\|m\a\|}
\ll \bigg(\f{t}{q_k}+1\bigg)q_k^2 \log q_k,  
\end{eqnarray*} 
and hence partial integration yields   
\begin{eqnarray*}
\sum_{q_k\leq m< q_{k+1}\atop q_k\nmid m} \f{\hat{H}(m)}{\|m\a\|}
&\ll& \int_{q_k}^{\infty} e^{-2\tau t} 
d\bigg\{\sum_{m\leq t\atop q_k\nmid m} \f{1}{\|m\a\|}\bigg\} \\ 
&\ll& q_k^2\log q_k \int_{q_k}^{\infty} t e^{-2\tau t}dt 
\ll e^{-\tau q_k}.    
\end{eqnarray*} 
On the other hand, 
by (\ref{qq+q|m}), 
\begin{eqnarray*}
\sum_{q_k\leq m< q_{k+1} \atop q_k|m}\f{1}{ \|m\a\|}
\ll q_{k+1}\log q_{k+1},   
\end{eqnarray*}
which together with (\ref{qk+1>eqk}) gives   
\begin{eqnarray*}
\sum_{q_k\leq m< q_{k+1} \atop q_k|m}\f{\hat{H}(m)}{ \|m\a\|}
\ll e^{-2 \tau q_k} q_{k+1}\log q_{k+1}
\ll e^{- \tau q_k} q_{k}. 
\end{eqnarray*} 
These prove that the series in (\ref{def/G=}) is absolutely convergent, and 
hence $G(x)$ is continuous.  In the same way we can prove that  
$G(x)$ is even smooth. 

Now we add $h$ to $H$ so that $h+H$ is smooth,  
and also 
\begin{eqnarray}
h(x)+H(x)=\{g(x+\a)+G(x+\a)\}-\{g(x)+G(x)\}. 
\end{eqnarray}
However $g(x)+G(x)$ cannot be a continuous function, since $G(x)$ is while $g(x)$ 
is not. 

In the following we want  to check that $h(x)+H(x)$ satisfies the upper bound 
and lower bound conditions 
(\ref{hhat/UPP}) and (\ref{hhat/LOW}) of our Theorem~\ref{thm2}.  
The $m$-th Fourier coefficient of $h+H$ is 
\begin{eqnarray*}
\left\{
\begin{array}{lll}
\hat{H}(m), \quad                       &\mbox{if  } m\not=q_k; \\   
\hat{H}(m)+\f{e(q_k\a)-1}{k}, \quad &\mbox{if  } m=q_k.    
\end{array}
\right. 
\end{eqnarray*}
The case $m\not=q_k$ is obvious. To check the case $m=q_k$, 
we apply (\ref{ratAPP}) and (\ref{qk+1>eqk}) to get   
\begin{eqnarray*}
\f{|e(q_k\a)-1|}{k}
\asymp \f{1}{k q_{k+1}}
\asymp \f{1}{ke^{\tau q_k}}, 
\end{eqnarray*}
which in combination with (\ref{hatH=}) yields   
\begin{eqnarray*}
\hat{H}(q_k)+\f{e(q_k\a)-1}{k}
\asymp \f{1}{ke^{\tau q_k}}. 
\end{eqnarray*}
Thus the Fourier coefficients of $h+H$ 
satisfy (\ref{hhat/UPP}) and (\ref{hhat/LOW}), and therefore 
Theorem~\ref{thm2} states that the M\"obius function is disjoint from the 
flow defined by $h+H$.

\section{Theorem~\ref{thm3}}
\setcounter{equation}{0}

For a review of preliminaries of nilmanifolds, the reader is referred to the 
Appendix II \S10. 

\subsection{Structure of affine  linear maps}
We begin with the structure of affine  linear maps.
By \S2.4 in particular Theorem~2.12
in Dani \cite{Dan}, any affine linear map $T$ of $G/\Gamma$ can be written as
\begin{eqnarray}\label{T=Tgs}
T=T_g \circ \overline\s
\end{eqnarray}
where $T_g$ is the action of $g\in G$ on $G/\Gamma$, $\s$ is an automorphism of $G$ such that
$\s(\Gamma)=\Gamma$, and $\overline\s: G/\Gamma\to G/\Gamma$ satisfies
$\overline\s(x\Gamma)=\s(x)\Gamma.$ 
It follows that
$$
T(x\Gamma)=T_g \{\overline\s(x\Gamma)\}= g \s(x)\Gamma,
$$
and by induction
\begin{eqnarray}\label{redDEF}
T^n(x\Gamma)=g\s(g)\cdots \s^{n-1}(g)\s^n(x)\Gamma.
\end{eqnarray}
We remark that (\ref{redDEF}) itself is not enough
to give a proof of Theorem~\ref{thm3}, since
the number of factors on the right-hand side of (\ref{redDEF})
depends on $n$.

\subsection{Application of zero entropy}
To prove Theorem~\ref{thm3}, we need the fact that the flow
$\mathscr{X}=(T, X)$ has zero entropy. The
main reference concerning the dynamics here is Dani's review article \cite{Dan}, Chapter 10.
In this setting the flow has zero
entropy if and only if it is quasi-unipotent. So the aim
is to prove Theorem~\ref{thm3} for such flows.

We need some words to clarify the definition.
Let $T=T_g \circ \overline\s$  be as in (\ref{T=Tgs}).
If all the eigenvalues of the differential $d\s:\fg\to\fg$ are of absolute value $1$, then we say
that $T$ and $\s$ are quasi-unipotent 
according to \S2.4 in Dani \cite{Dan}; this holds if and only if
all the eigenvalues are roots of unity. Further, when $G$ is simply connected,
the factor of $\s$ on $G/[G,G]$ is a linear automorphism and the proceeding
condition holds if and only if all the eigenvalues of the factor are roots of unity.

Let $\mathcal X = \{ X_1,\ldots,X_r\}$ be a basis for the Lie algebra $\fg$, and
for $x\in G$ let $\psi_{\exp}(x)=(u_1,\ldots,u_r)$ be the coordinates of the first kind.
Then $\s(x)$ can be computed by applying (\ref{IMPORT}) in Appendix II as follows:
\bna
\s(x)
&=& \s\{\exp(u_1X_1 + \cdots + u_rX_r)\} \\
&=& \exp\{(d\s)(u_1X_1 + \cdots + u_rX_r)\}.
\ena
Since $d\s$ is quasi-unipotent, we may assume that the matrix $U$
of $d\s$ under $\mathcal X$ is quasi-unipotent, and hence
\bna
(d\s) (u_1 X_1+\cdots+u_r X_r)
=(X_1,\ldots,X_r)U u,
\ena
where $u$ denotes the transpose of the row vector $(u_1,\ldots,u_r)$.
It follows that
\bna
(d\s)^n (u_1 X_1+\cdots+u_r X_r)
=(X_1,\ldots,X_r)U^n u,
\ena
and therefore
\bea\label{alphaINONE}
\s^n(x)
&=&\exp\{(d\s)^n (u_1 X_1+\cdots+u_r X_r)\}\nonumber\\
&=&\exp\{(X_1,\ldots,X_r)U^n u\}.
\eea

Since $U$ is quasi-unipotent, $U$ is a triangular
matrix with its diagonal entries being roots of unity.
It follows that there is a positive
integer $\nu$ such that
\begin{eqnarray}\label{Unu=I+N}
U^\nu=I+N
\end{eqnarray}
where $I$ is the identity matrix and $N$ is nilpotent. From now on we
let $\nu$ denote the least positive integer such that (\ref{Unu=I+N}) holds.
For any $n$, we can write $n=q\nu+l$ with $0\leq l\leq \nu-1$, and
therefore we can compute $U^n$ as
\bna
U^n
=U^{\nu q+l}=U^l(I+N)^q
=U^l \sum_{j=0}^{\min(q,r-1)}{q\choose j}N^j.
\ena
It follows that
\begin{eqnarray}\label{Unu=yt}
U^n u= y
\end{eqnarray}
where $y$ denotes the transpose of the row vector
$(y_{n1}(q),\ldots, y_{nr}(q))$ and each $y_{nk}(q)$ is a polynomial in $q$ with coefficients
depending on $U, x, \nu,$ and $l$. Of course $\deg y_{nk}\leq r-1$ for all $k=1,\ldots,r$.
Inserting (\ref{Unu=yt}) back to (\ref{alphaINONE}), we have
\begin{eqnarray}\label{alphaINO+}
\s^n(x)
=\exp\{y_{n1}(q)X_1+\cdots+y_{nr}(q)X_r\},
\end{eqnarray}
or, in the notation of $\psi_{\exp}$,
\begin{eqnarray}
\psi_{\exp}(\s^n(x))=(y_{n1}(q), \ldots, y_{nr}(q)).
\end{eqnarray}
Similar results holds for $\psi_{\exp}(\s^j(g))$ with $g\in G$ and
$j=1,\ldots,n-1$, that is
\begin{eqnarray}
\psi_{\exp}(\s^j(g))=(y_{j1}(q), \ldots, y_{jr}(q))
\end{eqnarray}
where each $y_{jk}(q)$ is a polynomial in $q$ with degree $\leq r-1$ and
with coefficients depending on $U, g, \nu,$ and $l$.
In the special case $j=0$ the above just reduces to the
coordinates $\psi_{\exp}(g)$ of $g$.

Now we apply Lemma~\ref{lem:3-2} in Appendix II $n$ times, so that the above analysis gives
\begin{eqnarray*}
\psi_{\exp}\{g\s(g)\cdots \s^{n-1}(g)\s^n(x)\}
=(Y_1(q), \ldots, Y_r(q))
\end{eqnarray*}
where $Y_1(q),\ldots,Y_r(q)$ are real polynomials in $q$ with
bounded degrees 
(which are actually $O_r(1)$ with the $O$-constant uniform in other parameters) 
and with their coefficients depending on $U, x, g, \nu,$ and $l$. 

By Lemma~\ref{lem:3-1} in Appendix II we can transform the coordinates of the first
kind to those for the second kind. Apply $\psi\circ\psi_{\exp}^{-1}$ to the above equality,
\begin{eqnarray*}
\psi\{g\s(g)\cdots \s^{n-1}(g)\s^n(x)\}
&=&(\psi\circ\psi_{\exp}^{-1})(Y_1(q), \ldots, Y_r(q)) \\
&=&(Z_1(q), \ldots, Z_r(q)),
\end{eqnarray*}
or
\begin{eqnarray}\label{alphaIN++}
g\s(g) \cdots \s^{n-1}(g)\s^n(x)
=\exp\{Z_1(q)X_1\}\cdots\exp\{Z_r(q)X_r\},
\end{eqnarray}
where $Z_1(q),\ldots,Z_r(q)$ are real polynomials in $q$ with
bounded degrees
and with their coefficients depending on
$U, x, g, \nu,$ and $l$.

For each $j=1,\ldots,r$ we may write
$$
Z_j(q)=c_{j\ell}q^{\ell}+\cdots+c_{j1}q+c_{j0},
$$
where $\ell=\deg Z_j$ and the coefficients $c_{jk}$'s are reals.
Recalling that $n=q\nu+l$ with $0\leq l\leq \nu-1$, we may write $Z_j(q)$ as
a polynomial in $n$ as follows
$$
Z_j(q)=c'_{j\ell}n^{\ell}+\cdots+c'_{j1}n+c'_{j0},
$$
where $c'_{jk}$'s are real coefficients depending on $U, g, x, \nu,$ and $l$.
It follows that
\bna
\exp\{Z_j(q)X_j\}
=\exp(c'_{j\ell}X_j n^{\ell})\cdots\exp(c'_{j1}X_j n)\exp(c'_{j0})
=b_{j\ell}^{n^{\ell}}\cdots b_{j1}^{n}b_{j0}
\ena
with $b_{j\ell}=\exp(c'_{j\ell}X_j)$ etc. Inserting these into (\ref{alphaIN++}),
we see that
\begin{eqnarray*}
g\s(g) \cdots \s^{n-1}(g)\s^n(x)
=b_1^{h_1(n)}\cdots b_k^{h_k(n)}
\end{eqnarray*}
where $b_1,\ldots, b_k\in G$ and
$h_1,\ldots, h_k$ are integral polynomials in $n$. Here it is important to
note that $k$ does not depend on $n$. Thus (\ref{redDEF}) becomes
\bea\label{redDEF+}
T^n(x\Gamma)
&=& g\s(g)\cdots \s^{n-1}(g)\s^n(x)\Gamma \nonumber\\
&=& b_1^{h_1(n)}\cdots b_k^{h_k(n)}\G.
\eea
Compared with (\ref{redDEF}), this has the advantage that
the number $k$ of factors on the right-hand side is
independent of $n$. This fact will be important for the following lemma to hold. 

\begin{lemma}\label{lem:4-1}\footnote{As pointed out to us by Tao this Lemma can be deduced directly from 
Theorem~1.1 of  
\cite{GreTao2} using the constructions with disconnnected nilmanifolds as in \cite{Leb}.} 
Let $\nu$ be a positive integer and $0\leq l<\nu$. 
Let $G/\Gamma$ be a nilmanifold and
$f: G/\Gamma\to [-1,1]$ a Lipschitz function. 
Let $b_1,\ldots, b_k\in G$ and $h_1,\ldots, h_k$ be
integral polynomials in $n$, where $k$ does not depend on $n$.
Then, for any $A>0$,
\bea\label{GgtThm}
\sum_{n\leq N\atop n\equiv l(\bmod \nu)}\mu(n)f\big(b_1^{h_1(n)}\cdots b_k^{h_k(n)}\G\big)
\ll N\log^{-A}N
\eea
where the implied constant depends on $G,\G, T, f, x,\nu$, and $A$.  
\end{lemma}

Lemma~\ref{lem:4-1} can be established in the same way as 
Theorem~1.1 in Green-Tao \cite{GreTao2}, where the case $\nu=l=1$ is handled. 
Now a proof of Theorem~\ref{thm3} is 
immediate.

\begin{proof}[Proof of Theorem~\ref{thm3}]
Recall that $\nu$ is the least positive integer satisfying (\ref{Unu=I+N}), that is 
$\nu$ is fixed.  
Then each $n\in \Bbb N$ can be written as
$n=\nu q+l$ with $0\leq l\leq \nu-1$, and our original sum takes the form
\begin{eqnarray*}
\sum_{n\leq N}\mu(n)f(T^n (x\G))
&=&\sum_{l=0}^{\nu-1}\sum_{n\leq N\atop n\equiv l(\bmod \nu)}
\mu(n)f(T^n (x\G))\nonumber\\
&=&\sum_{l=0}^{\nu-1}\sum_{n\leq N\atop n\equiv l(\bmod \nu)}\mu(n)f\big(b_1^{h_1(n)}\cdots b_k^{h_k(n)}\G\big)
\end{eqnarray*}
by (\ref{redDEF+}). Applying Lemma~\ref{lem:4-1} to the last sum over $n$, we get
\begin{eqnarray*}
\sum_{n\leq N}\mu(n)f(T^n (x\G))\ll N\log^{-A}N
\end{eqnarray*}
where the implied constant depends on $G, \G, T, f, x, \nu$, and $A$.
Theorem~\ref{thm3} is proved. 
\end{proof}

\section{Appendix I: Disjointness of $\mu$ from double exponential functions} 
\setcounter{equation}{0}
The method of the paper proves, actually more than that, 
the following trigonometric analogue of Hua's Lemma~\ref{lem:VDH}.  

\begin{proposition}\label{prop:TriHua} 
Let $D$ be fixed, and 
\begin{eqnarray}
\phi(n)= \sum_{1\leq |m|\leq D} \l_m e(m\th n)
\end{eqnarray}
be a real trigonometric polynomial. 
Then as $N\to \infty$ we have 
\begin{eqnarray}\label{App1/SoN}
\sum_{n\leq N}\mu(n) e(\phi(n))=o(N)  
\end{eqnarray}
uniformly in $\l=(\l_{-D}, \ldots, \l_{-1}, \l_1, \ldots, \l_D)$ and $\th$. 
\end{proposition}

\section{Appendix II: preliminaries on nilmanifolds}
\setcounter{equation}{0}

\subsection{Nilmanifolds}
Let $G$ be a connected, simply connected
nilpotent Lie group of dimension $r$. A {\it filtration}
$G_{\bullet}$ on $G$ is a sequence of closed connected groups
\begin{eqnarray*}
G=G_0= G_1 \supset \cdots \supset G_d \supset G_{d+1} = \{\mbox{\rm id}_G\}
\end{eqnarray*}
with the property that $[G_j,G_k]\subset G_{j+k}$ for all $j,k\geq 0$. Here $[H,K]$
denotes the commutator group of $H$ and $K$. The {\it degree} $d$ of $G_{\bullet}$ is
the least integer such that $G_{d+1} = \{\mbox{\rm id}_G\}$. We say that $G$ is
{\it nilpotent} if $G$ has a filtration.
If $\G$ is a discrete and cocompact subgroup of $G$,
then $G/\Gamma=\{g\G:g\in G\}$ is called a {\it nilmanifold}.
We write $r=\dim G$ and $r_j=\dim G_j$ for $j=1,\ldots,d$. If
a filtration $G_{\bullet}$ of degree $d$ exists then the lower central
series filtration defined by
$$
G=G_0=G_1, \quad G_{j+1}=[G,G_j]
$$
terminates with $G_{s+1}=\{\mbox{\rm id}_G\}$ for some $s\leq d$. The least such
$s$ is called the {\it step} of the nilpotent Lie group $G$.

\subsection{Connections with Lie Algebra}
Let $\fg$ be the Lie
algebra of $G$, and let $\exp: \fg\to G$ and $\log: G\to\fg$
be the exponential and logarithm maps, which are both diffeomorphisms.
We can also define the $1$-parameter subgroup
$(g^t)_{t\in\Bbb R}$ associated to an element $g\in G$, and thus
$$
\exp(X)^t = \exp(tX)
$$
for all $X\in {\mathfrak g}$ and $t\in \Bbb R.$
For an automorphism $\s$ of $G$
we denote by
$d\s:\fg\to\fg$ the differential of $\s$.
Then we have
\bea\label{IMPORT}
\s(\exp(X))=\exp\{(d\s)(X)\} \quad\mbox{for any } X\in \fg.
\eea
These maps are illustrated below:
$$
\begin{array}{rcl}
G        & {\s\atop\longrightarrow} & G \\
\exp \uparrow &             & \downarrow \log\\
\fg      & {\longrightarrow\atop d\s} & \fg \\
\end{array}
$$

\subsection{Coordinates of the first and second kind.}
Now we give the notion of coordinates of the first and second
kinds. Let $\mathcal X = \{ X_1,\ldots,X_r\}$ be a basis for the Lie algebra
$\fg$. If
$$
g = \exp(u_1X_1 + \cdots + u_rX_r),
$$
then we say that $(u_1,\ldots,u_r)$ are the {\it coordinates of the first kind} or
exponential coordinates for $g$ relative to the basis $\mathcal X$.
We write $(u_1,\ldots,u_r) = \psi_{\exp}(g)$.
If
$$
g = \exp(v_1X_1)\cdots \exp(v_rX_r),
$$
then we say that $(v_1,\ldots,v_r)$ are the {\it coordinates of the second kind}
for $g$ relative to $\mathcal X$, and we write $(v_1,\ldots,v_r) = \psi(g)$.
The {\it height} of a reduced rational number $\f{a}{b}$ is defined to be
$\max\{|a|,|b|\}$. The basis $\mathcal X$ is said to be $Q$-{\it rational} if all the structure constants $c_{ijk}$
in the relations
$$
[X_i,X_j]=\sum_{k}c_{ijk}X_k
$$
are rationals of height at most $Q.$ 

The following lemmas describes the connection between the two types of
coordinate systems; they are Lemmas~A.2 and A.3 in Green-Tao \cite{GreTao1}.

\begin{lemma}\label{lem:3-1}
Let $\mathcal X$ be a basis for $\fg$ such that
\bea\label{nestPROP}
[\fg,X_j]\subset \mbox{\rm Span}(X_{j+1},\ldots,X_r)
\eea
for $j=1,\ldots,r-1.$ Then the compositions  $\psi_{\exp}\circ \psi^{-1}$ and
$\psi\circ \psi_{\exp}^{-1}$ are both polynomial maps on ${\Bbb R}^r$
with bounded degree. If $\mathcal X$ is $Q$-rational then all the coefficients of these
polynomials are rational of height at most $Q^C$ for some constant $C>0$.
\end{lemma}

\begin{lemma}\label{lem:3-2}
Let $\mathcal X$ be a basis for $\fg$ satifying (\ref{nestPROP}). Let $x, y\in G$, and suppose that
$\psi(x) =(u_1,\ldots,u_r)$ and $\psi(y) =(v_1,\ldots,v_r)$. Then
$$
\psi_{\exp}(x) = (u_1, u_2 + R_1(u_1),\ldots, u_r + R_{r-1}(u_1,\ldots,u_{r-1})),
$$
where each $R_j: {\Bbb R}^j\to {\Bbb R}$ is a polynomial of bounded degree. Also,
$$
\psi_{\exp}(xy)
=(u_1 + v_1, u_2 + v_2 + S_1(u_1, v_1),\ldots, u_r + v_r + S_{r-1}(u_1,\ldots, u_{r-1}, v_1,\ldots, v_{r-1})),
$$
where each $S_j: {\Bbb R}^j\times {\Bbb R}^j\to {\Bbb R}$
is a polynomial of bounded degree.

Let $Q\geq 2$. If $\mathcal X$ is $Q$-rational then all
the coefficients of the polynomials $R_j, S_j$ are rationals of height $Q^C$ for some constant $C>0$.
\end{lemma}

\medskip 
\noindent 
{\bf Acknowledgements.} We thank Terry Tao for his insightful comments 
on an earlier version of this paper. While working on this project the first author made  
multiple visits to the Institute for Advanced Study and Princeton University 
in 2010-2013, and it is a pleasure to record his gratitude to both institutions. 
The first author is supported by the 973 Program, 
NSFC 11031004, and IRT 1264. The second author is supported by an NSF grant.

\end{document}